\documentclass[12pt,twoside,leqno]{article}
\usepackage[T1]{fontenc}
\usepackage{amsfonts}
\usepackage{amsmath,amsthm}
\usepackage{amssymb,latexsym}
\usepackage{mathrsfs} 

\theoremstyle{plain}
\newtheorem{theorem}{Theorem}[section]
\newtheorem{lemma}[theorem]{Lemma}
\newtheorem{proposition}[theorem]{Proposition}
\newtheorem{corollary}[theorem]{Corollary}

\theoremstyle{definition}

\newtheorem{definition}[theorem]{Definition}
\theoremstyle{remark}
\newtheorem{remark}[theorem]{Remark}

\newtheorem{problem}[theorem]{Problem}

\DeclareMathOperator{\inter}{{\rm int}}
\DeclareMathOperator{\cl}{{\rm cl}}

\DeclareMathOperator{\id}{{\rm id}}

\begin{document}
\title{Characterizations of $\mathbb{N}$-compactness and realcompactness via ultrafilters in the absence of the axiom of choice}
\author{ \textbf{AliReza Olfati and Eliza Wajch}\\
\\
Department of Mathematics, Faculty of Basic Sciences,\\ 
Yasouj University, Daneshjoo St.,  Yasouj 75918-74934, Iran\\
E-mail: alireza.olfati@yu.ac.ir\\
School of Mathematics\\
Institute for Research in Fundamental Sciences (IPM)\\
P.O. Box:19395-5746, Tehran\\
\\
Institute of Mathematics,
Faculty of Exact and Natural Sciences,\\
University of Siedlce,\\
ul. 3 Maja 54, 08-110 Siedlce, Poland\\
E-mail: eliza.wajch@gmail.com}
\maketitle
\begin{abstract}
This article concerns the Herrlich-Chew theorem stating that a Hausdorff zero-dimensional space is $\mathbb{N}$-compact if and only if every clopen ultrafilter with the countable intersection property in this space is fixed. It also concerns Hewitt's theorem stating that a Tychonoff space is realcompact if and only if every $z$-ultrafilter with the countable intersection property in this space is fixed. The axiom of choice was involved in the original proofs of these theorems. The aim of this article is to show that the Herrlich-Chew theorem is valid in $\mathbf{ZF}$, but it is an open problem if Hewitt's theorem can be false in a model of $\mathbf{ZF}$. It is proved that Hewitt's theorem is true in every model of $\mathbf{ZF}$ in which the countable axiom of multiple choice is satisfied. A modification of Hewitt's theorem is given and proved true in $\mathbf{ZF}$. Several applications of the results obtained are shown.

 
\end{abstract}

\thanks{The first author declares that the research funding for this project was provided by the Institute for Research in Fundamental Sciences (IPM) Grants Committee (Award No. 1403030024)}

\maketitle

\section{Introduction}
\label{intro}

The standard Zermelo-Fraenkel system of axioms $\mathbf{ZF}$ is the set-theoretic framework for this article. This means that we conduct research without assuming that the axiom of choice, usually denoted by $\mathbf{AC}$, is true. The system $\mathbf{ZF+AC}$ is commonly denoted by $\mathbf{ZFC}$. 

In what follows, our notation and terminology, if not introduced here, is standard. The set $\omega$ of all finite ordinal numbers (of von Neumann) is identified with the set of all non-negative integers of the real line $\mathbb{R}$. The empty set $\emptyset$ as the smallest ordinal number is denoted by $0$. For every ordinal number $\alpha$, $\alpha+1=\alpha\cup\{\alpha\}$. For our convenience, we put $\mathbb{N}=\omega\setminus\{0\}$ and call every element of $\mathbb{N}$ a \emph{natural number} or a \emph{positive integer}.

For a set $X$, $\mathcal{P}(X)$ denotes the power set of $X$, $[X]^{<\omega}$ is the family of all finite subsets of $X$, $[X]^{\leq\omega}$ is the family of all countable subsets of $X$. If $\mathcal{A}\subseteq\mathcal{P}(X)$, then
\[\mathcal{A}_{\delta}:=\{\bigcap\mathcal{U}: \mathcal{U}\in [\mathcal{A}]^{\leq\omega}\setminus\{\emptyset\}\}.\]

Given a family $\mathcal{X}=\{X_t: t\in T\}$ of sets and an element $t_0\in T$, we denote by $\pi_{t_0}$ the projection of the product $\prod_{t\in T}X_t$ into $X_{t_0}$. If the product $\prod_{t\in T}([X_t]^{<\omega}\setminus\{\emptyset\})$ is non-empty, then it is said that $\mathcal{X}$ has a multiple choice function, and every element of $\prod_{t\in T}([X_t]^{<\omega}\setminus\{\emptyset\})$ is called a \emph{multiple choice function} of $\mathcal{X}$. 

  We usually denote topological and metric spaces with boldface letters, and their underlying sets with lightface letters. 

Given a topological space $\mathbf{X}=\langle X, \tau\rangle$ and a subset $P$ of $X$, $\mathbf{P}$ denotes the (topological) subspace of $\mathbf{X}$ with the underlying set $P$. That is, $\mathbf{P}=\langle P, \tau\vert_P \rangle$ where $\tau\vert_P=\{U\cap P: u\in\tau\}$; however, if this is not misleading, the subspace $\mathbf{P}$ of $\mathbf{X}$ can be denoted by $P$. The closure of $P$ in $\mathbf{X}$ is denoted by $\cl_{\mathbf{X}}(P)$, and the interior of $P$ in $\mathbf{X}$ is denoted by $\inter_{\mathbf{X}}(P)$. 

If this is not misleading, we denote by $\mathbb{R}$ not only the set of all real numbers but also the topological space $\langle\mathbb{R}, \tau_{nat}\rangle$ where $\tau_{nat}$ is the natural topology induced by the standard linear order $\leq$ of the set $\mathbb{R}$. 

Let $\mathbf{X}=\langle X, \tau_X\rangle$ and $\mathbf{Y}=\langle Y, \tau_Y\rangle$ be topological spaces. The set of all $\langle\tau_X, \tau_Y\rangle$-continuous mappings $f:X\to Y$ is denoted by $C(\mathbf{X}, \mathbf{Y})$. A \emph{continuous mapping} $f:\mathbf{X}\to\mathbf{Y}$ is a $\langle\tau_X, \tau_Y\rangle$-continuous mapping $f: X\to Y$. The set $C(\mathbf{X}, \mathbb{R})$ is denoted by $C(\mathbf{X})$. For $f\in C(\mathbf{X})$, the set $Z(f):=f^{-1}[\{0\}]$ is the \emph{zero-set} of $f$. We define $\mathcal{Z}(\mathbf{X}):=\{Z(f): f\in C(\mathbf{X})\}$ and $\mathcal{Z}^c(\mathbf{X}):=\{ U\in\mathcal{P}(X): X\setminus U\in \mathcal{Z}(\mathbf{X})\}$. Members of $\mathcal{Z}(\mathbf{X})$ are called \emph{zero-sets} of $\mathbf{X}$. Members of $\mathcal{Z}^c(\mathbf{X})$ are called \emph{cozero-sets} in $\mathbf{X}$. A subset $A$ of $\mathbf{X}$ is called \emph{clopen} if both $A$ and $X\setminus A$ are open in $\mathbf{X}$. The family of all clopen sets in $\mathbf{X}$ is denoted by $\mathcal{CO}(\mathbf{X})$. We denote by $\mathcal{CO}_{\delta}(\mathbf{X})$ the family $\mathcal{CO}(\mathbf{X})_{\delta}$. That is, 
\[\mathcal{CO}_{\delta}(\mathbf{X})=\{\bigcap \mathcal{U}:\mathcal{U}\in[\mathcal{CO}(\mathbf{X})]^{\leq\omega}\setminus\{\emptyset\}\}.\] 
\noindent Members of $\mathcal{CO}_{\delta}(\mathbf{X})$ are called $c_{\delta}$-sets in $\mathbf{X}$. Filters (respectively, ultrafilters) in the family $\mathcal{CO}(\mathbf{X})$ are called \emph{clopen filters} (respectively, \emph{clopen ultrafilters}) in $\mathbf{X}$. Filters (respectively, ultrafilters) in the family $\mathcal{Z}(\mathbf{X})$ are called $z$-\emph{filters} (respectively, $z$-\emph{ultrafilters}) in $\mathbf{X}$.  Filters (respectively, ultrafilters) in the family $\mathcal{CO}_{\delta}(\mathbf{X})$ are called $c_{\delta}$-\emph{filters} (respectively, $c_{\delta}$-\emph{ultrafilters}) in $\mathbf{X}$. A filter $\mathcal{F}$ in a family $\mathcal{R}$ of subsets of $\mathbf{X}$ is called a \emph{fixed filter} if $\bigcap\mathcal{F}\neq\emptyset$; otherwise, $\mathcal{F}$ is called a \emph{free filter}. It is said that a family $\mathcal{F}$ of subsets of $X$ has the \emph{countable intersection property} if, for every $\mathcal{U}\in[\mathcal{F}]^{\leq\omega}\setminus\{\emptyset\}$, $\bigcap\mathcal{U}\neq\emptyset$.

In the following definition, we modify the concept of the countable intersection property of a family of zero-sets.

\begin{definition}
\label{s1:d1} 
	\label{s3;d6}
	Let $\mathbf{X}$ be a topological space and let $\mathcal{F}\subseteq\mathcal{Z}(\mathbf{X})$. We say that:
	\begin{enumerate}
		\item[(i)] $\mathcal{F}$ is \emph{functionally accessible} if there exists a family $\{f_Z: Z\in\mathcal{F}\}$ of functions from $C(\mathbf{X})$ such that, for every $Z\in\mathcal{F}$, $Z(f_Z)=Z$;
		\item[(ii)]$\mathcal{F}$ has the \emph{weak countable intersection property} if, for every non-empty functionally accessible countable subfamily $\mathcal{F}^{\ast}$ of $\mathcal{F}$, $\bigcap\mathcal{F}^{\ast}\neq\emptyset$. 
	\end{enumerate}
\end{definition}

We are concerned with the classes of spaces defined as follows.

\begin{definition}
\label{s1:d2} A topological space $\mathbf{X}$ is called:
\begin{enumerate}
\item[(i)] $\mathbb{N}$-\emph{compact} if there exists a non-empty set $J$ such that $\mathbf{X}$ is homeomorphic with a closed subspace of $\mathbb{N}^J$;
\item[(ii)] \emph{realcompact} if there exists a non-empty set $J$ such that $\mathbf{X}$ is homeomorphic with a closed subspace of $\mathbb{R}^J$;
\item[(iii)] \emph{completely regular} if $\mathcal{Z}^c(\mathbf{X})$ is a base for $\mathbf{X}$; completely regular $T_1$-spaces are called \emph{Tychonoff spaces};
\item[(iv)] \emph{zero-dimensional} if $\mathcal{CO}(\mathbf{X})$ is a base for $\mathbf{X}$;
\item[(v)] \emph{strongly zero-dimensional} if $\mathbf{X}$ is completely regular and, for every pair $A,B$ of disjoint zero-sets in $\mathbf{X}$, there exists $C\in\mathcal{CO}(\mathbf{X})$ such that $A\subseteq C\subseteq X\setminus B$.
\end{enumerate}
\end{definition}

Realcompact spaces were called $Q$-spaces by Hewitt in \cite{hewitt}. Shirota showed in \cite{sh} that the class of Hewitt's $Q$-spaces is exactly the class of $e$-complete spaces in $\mathbf{ZFC}$.  Already Hewitt proved in \cite{hewitt} that the following theorem is true in $\mathbf{ZFC}$.

\begin{theorem}
\label{s1:t3}
(Hewitt's theorem; see \cite[Theorem 50]{hewitt}.) A Tychonoff space $\mathbf{X}$ is realcompact if and only if every $z$-ultrafilter with the countable intersection property in $\mathbf{X}$ is fixed.
\end{theorem}

That Theorem \ref{s1:t3} is true in $\mathbf{ZFC}$ was shown also by Shirota in \cite[Theorem 1]{sh}. Herrlich proved in \cite{her}, and independently,  Chew proved in \cite{Chew} that the following theorem is true in $\mathbf{ZFC}$.

\begin{theorem}
\label{s1:t4}
(Herrlich-Chew theorem; see \cite[Satz 6.2, p. 251]{her} and  \cite[Theorem A]{Chew}.) A zero-dimensional $T_1$-space is $\mathbb{N}$-compact if and only if every clopen ultrafilter with the countable intersection property in $\mathbf{X}$ is fixed.
\end{theorem}

To the best of our knowledge, the axiom of choice is involved in all already published proofs of Theorems \ref{s1:t3} and \ref{s1:t4}, so their known proofs are valid in $\mathbf{ZFC}$ but are incorrect in $\mathbf{ZF}$. The main aim of this article is to investigate whether Theorems \ref{s1:t3} and \ref{s1:t4} can be proved also in $\mathbf{ZF}$. To do this well, we need to apply to some results the weaker forms of $\mathbf{AC}$, defined as follows.

\begin{definition}
\label{s1:d5}
\begin{itemize}
\item $\mathbf{BPI}$ (the Boolean Prime Ideal Theorem, \cite[Form 14]{hr}): Every Boolean algebra has a prime ideal.
\item $\mathbf{CMC}$ (the countable axiom of multiple choice, \cite[Form 126]{hr}): Every countable family $\mathcal{A}=\{A_n: n\in\mathbb{N}\}$ of non-empty sets has a multiple choice function. 
\item $\mathbf{A}(\delta\delta)$: For every set $X$ and every family $\mathcal{A}$ of subsets of $X$ such that $\mathcal{A}$ is closed under finite unions and finite intersections, for every family $\{\mathcal{U}_n: n\in\omega\}$ of non-empty countable subfamilies of $\mathcal{A}$, there exists a family $\{V_{m,n}: m,n\in\omega\}$ of members of $\mathcal{A}$ such that, for every $n\in\omega$, $\bigcap\mathcal{U}_n=\bigcap_{m\in\omega}V_{m,n}$.(Cf. \cite[Definition 1.7(4)]{kow}.)
\end{itemize}
\end{definition}

\begin{remark}
\label{s1:r6}
(a) The form $\mathbf{A}(\delta\delta)$ was introduced to mathematics in \cite[Definition 1.7(4)]{kow}. By \cite[Theorem 2.9(ii) and Remark 2.20]{kow}, it holds in $\mathbf{ZF}$ that $\mathbf{CMC}$ implies $\mathbf{A}(\delta\delta)$, and this implication is not reversible.\medskip

(b) One can check in \cite{hr} that, for instance, Cohen's original model $\mathcal{M}1$ in \cite{hr} is a model of $\mathbf{ZF+BPI+\neg CMC}$, Feferman's model $\mathcal{M}2$ in \cite{hr} is a model of $\mathbf{ZF+CMC+\neg BPI}$, Pincus' model IV (model $\mathcal{M}40(\kappa)$ in \cite{hr}) is a model of $\mathbf{ZF+BPI+CMC}$, Monro's model III (model $\mathcal{M}37$ in \cite{hr}) is a model of $\mathbf{ZF+\neg  BPI+\neg CMC}$. The axiom of choice is false in all these models.\medskip

(c) We recall that $\mathbf{BPI}$ is equivalent to the following statements:
\begin{itemize}
\item ( \cite[Form 14 J]{hr}) Every product of compact Hausdorff spaces is compact.
\item (\cite[Form 14 K]{hr}) Every Cantor cube is compact.
\end{itemize}
\end{remark}

The main results  of our article concerning the question if Hewitt's theorem or Chew's theorem (or both the theorems) can be proved in $\mathbf{ZF}$ are summarized in the following theorem.

\begin{theorem}
\label{s1:t7}
\begin{enumerate}
\item[(i)] The Herrlich-Chew theorem (Theorem \ref{s1:t4}) is provable in $\mathbf{ZF}$.
\item[(ii)] For every Tychonoff space $\mathbf{X}$, it is true in $\mathbf{ZF}$ that $\mathbf{X}$ is realcompact if and only if every $z$-ultrafilter with the weak countable intersection property in $\mathbf{X}$ is fixed.
\item[(iii)] Hewitt's theorem (Theorem \ref{s1:t3}) is valid in every model of $\mathbf{ZF+CMC}$.
\item[(iv)]  It is true in $\mathbf{ZF}$ that, for every zero-dimensional $T_1$-space $\mathbf{X}$, $\mathbf{CMC}$ implies that $\mathbf{X}$ is $\mathbb{N}$-compact if and only if every $c_{\delta}$-ultrafilter with the countable intersection property in $\mathbf{X}$ is fixed.
\end{enumerate}
\end{theorem}

Theorem \ref{s1:t7}(i) is proved in Section \ref{s3} (see Theorem \ref{s3:t2}), and Theorem \ref{s1:t7}(iv) is proved in Section \ref{s4} (see Theorem \ref{s4:t2}).  Our proofs of items (ii) and (iii) of Theorem \ref{s1:t7} are included in Section \ref{s5} (see Theorem \ref{s5:t3} and Corollary \ref{s5:c4}). Several applications of Theorem \ref{s1:t7} are shown in Section \ref{s6}. The list of open problems concernig the results obtained here is given in Section \ref{s7}.

Before we pass to detailed proofs of items (i)--(iv) of Theorem \ref{s1:t7}, we need to establish some general facts about $\mathbf{E}$-compact spaces and Hewitt $\mathbf{E}$-compact extensions of $\mathbf{E}$-completely regular spaces in $\mathbf{ZF}$ in Section \ref{s2}.

\section{Basic facts about $\mathbf{E}$-compactness in $\mathbf{ZF}$}
\label{s2}

In $\mathbf{ZFC}$, a systematic study of $\mathbf{E}$-completely regular and $\mathbf{E}$-compact spaces, as a common generalization of realcompact and $\mathbb{N}$-compact spaces, in the sense of Definition \ref{s2:d1} given below, was started by Engelking and Mr\'owka in \cite{enmr}.

\begin{definition}
	\label{s2:d1}
	(Cf. e.g., \cite{enmr}, \cite{mr1}--\cite{mr5}.) Let $\mathbf{E}$ be a given topological space. A topological space $\mathbf{X}$ is called  $\mathbf{E}$-\emph{completely regular} (respectively, $\mathbf{E}$-\emph{compact}) if there exists a non-empty set $J$ such that $\mathbf{X}$ is homeomorphic with a subspace (respectively, a closed subspace) of $\mathbf{E}^J$. 
\end{definition}

The class of $\mathbb{R}$-completely regular spaces is the class of Tychonoff spaces, the class of $\mathbb{N}$-completely regular spaces is the class of zero-dimensional $T_1$-spaces, and the class of $\mathbb{R}$-compact spaces is the class of realcompact spaces. 

We need to apply realcompact and $\mathbb{N}$-compact extensions, so let us recall the following definitions.

\begin{definition}
	\label{s2:d2}
	Let $\mathbf{X}$, $\mathbf{Y}$ and $\mathbf{E}$ be  topological spaces.
	\begin{enumerate}
		\item[(a)] An \emph{extension} of $\mathbf{X}$ is an ordered pair $\langle\mathbf{Y}, h\rangle$ where  $h$ is a homeomorphic embedding of $\mathbf{X}$ into $\mathbf{Y}$ such that $h[X]$ is dense in $\mathbf{Y}$.
		\item[(b)] If, for $i\in\{1,2\}$,  $\langle\mathbf{Y}_i, h_i\rangle$ are extensions of $\mathbf{X}$, then:
		\begin{enumerate}
			\item[(i)] we write $\langle \mathbf{Y}_1, h_1\rangle\leq \langle \mathbf{Y}_2, h_2\rangle$ if there exists a continuous surjection $f: \mathbf{Y}_2\to \mathbf{Y}_1$ such that $f\circ h_2=h_1$;
			\item[(ii)] we write $\langle \mathbf{Y}_1, h_1\rangle\thickapprox\langle\mathbf{Y}_2, h_2\rangle$ and say that $\langle \mathbf{Y}_1, h_1\rangle$ and $\langle\mathbf{Y}_2, h_2\rangle$ are \emph{equivalent extensions} of $\mathbf{X}$ if there exists a homeomorphism $h: \mathbf{Y}_1\to\mathbf{Y}_2$ such that $h\circ h_1=h_2$.
		\end{enumerate}
		\item[(c)] If $\mathcal{P}$ is a topological property, then we say that an extension $\langle\mathbf{Y}, h\rangle$ has $\mathcal{P}$ if $\mathbf{Y}$ has $\mathcal{P}$. Compact extensions of $\mathbf{X}$ are called \emph{compactifications} of $\mathbf{X}$. $\mathbb{R}$-compact extensions are called \emph{realcompactifications}.
	\end{enumerate}
\end{definition}

\begin{remark}
	\label{s2:r3} 
	Let $\mathbf{X}$ be a topological space.
	\begin{enumerate}
		\item[(i)]Suppose $\langle \mathbf{Y}_1, h_1\rangle$ and $\langle \mathbf{Y}_2, h_2\rangle$ are Hausdorff extensions of $\mathbf{X}$. Then $\langle \mathbf{Y}_1, h_1\rangle\thickapprox \langle\mathbf{Y}_2, h_2\rangle$ if and only if $\langle \mathbf{Y}_1, h_1\rangle\leq \langle\mathbf{Y}_2, h_2\rangle$ and $\langle \mathbf{Y_2}, h_2\rangle\leq\langle \mathbf{Y}_1, h_1\rangle$.  
		
		\item[(ii)] For an extension $\langle\mathbf{Y}, h\rangle$ of $\mathbf{X}$, the space $\mathbf{X}$ is identified with the subspace $h[X]$ of $\mathbf{Y}$, and $h$ is identified  with the identity mapping $\id_{X}$ on $X$. The space $\mathbf{Y}$ and the extension $\langle \mathbf{Y}, h\rangle$ can be both denoted by $h\mathbf{X}$.
	\end{enumerate}
\end{remark} 

\begin{definition}
	\label{s2:d4}
	Suppose that $\mathbf{X}$ and $\mathbf{E}$ are Hausdorff spaces such that $\mathbf{X}$ is $\mathbf{E}$-completely regular.
	An $\mathbf{E}$-compact extension $\langle \mathbf{Y}, h\rangle$ of $\mathbf{X}$ is called a \emph{Hewitt} $\mathbf{E}$-\emph{compact extension} of $\mathbf{X}$ if it satisfies  the following condition:
	$$(\forall f\in C(\mathbf{X}, \mathbf{E}))(\exists \tilde{f}\in C(\mathbf{Y}, \mathbf{E}))\quad\tilde{f}\circ h=f.$$
\end{definition}

Many results on $\mathbf{E}$-compactness and $\mathbf{E}$-compact extensions, in particular, on realcompactness and Hewitt realcompactifications in $\mathbf{ZFC}$ are well-known, for instance, from \cite{en}, \cite{enmr}, \cite{gj},  \cite{hewitt}, \cite{mr0}--\cite{mr5}, \cite{nik} and  \cite{PorterWoods}. We especially recommend \cite[Chapter 3.11]{en} as a collection of basic results on realcompatness and Hewitt realcompactifications in $\mathbf{ZFC}$. Not all proofs of the results we need are given in $\mathbf{ZF}$ in the published literature. We should also mention that Comfort gave in \cite{com} a rather complicated $\mathbf{ZF}$-proof that a non-empty Tychonoff space is realcompact if and only if every real ideal of $C(\mathbf{X})$ is fixed. In addition, Comfort observed in \cite{com} that it is true in $\mathbf{ZF}$ that every Tychonoff space has its Hewitt realcompactification. However, to the best of our knowledge, our unpublished preprint \cite{ow1} is the first extensive work on $\mathbf{E}$-compactness (in particular, on realcompactness and $\mathbb{N}$-compactness) in $\mathbf{ZF}$. In the present article, we do not wish to apply real ideals  of rings of continuous real functions to $\mathbf{ZF}$-proofs of our main results. Characterizations of  $\mathbb{N}$-compactness via real ideals of certain rings of continuous real functions in $\mathbf{ZF}$ can be found in our preprint \cite{ow1} and will be a topic of our future work. We omit most of the preliminary, unpublished results included in \cite{ow1}, but show here only the ones we need to prove the main theorems announced in Section \ref{intro}. 

\begin{definition}
\label{s2:d5}
For topological spaces $\mathbf{X}$ and $\mathbf{E}$, $\mathcal{E}(\mathbf{X}, \mathbf{E})$ denotes the family of all subsets $F$ of $C(\mathbf{X}, \mathbf{E})$ such that the evaluation map $e_F: X\to E^{F}$, defined by:
	\[(\forall x\in X)(\forall f\in F) e_F(x)(f):=f(x),\]
is a homeomorphic embedding. For every $F\in\mathcal{E}(\mathbf{X}, \mathbf{E})$, the extension of $\mathbf{X}$ generated by $F$ is denoted by $e_F\mathbf{X}$ and defined as follows:
\[e_F\mathbf{X}:=\cl_{\mathbf{E}^F}(e_F[X]).\]
\end{definition}

\begin{remark}
\label{s2:r6}
Let $\mathbf{E}$ be a non-empty topological space. Then a topological space $\mathbf{X}$ is $\mathbf{E}$-completely regular if and only if $C(\mathbf{X}, \mathbf{E})\in\mathcal{E}(\mathbf{X},\mathbf{E})$.
\end{remark}

We recall the following well-known result.

\begin{proposition}
\label{s2:p7}
$[\mathbf{ZF}]$  Let $\mathbf{Y}$ be a Hausdorff space and let $D$ be a dense subset of a topological space $\mathbf{X}$. If $f,g\in C(\mathbf{X}, \mathbf{Y})$ are such that $f\upharpoonright D= g\upharpoonright D$, then $f=g$.
\end{proposition}

The proof of the following theorem is a modification of \cite[proof of Theorem 3.11.3]{en}. We include it for clarity and completeness.

\begin{theorem}
	\label{s2:t8}
	$[\mathbf{ZF}]$  Let $\mathbf{E}$ be a non-empty Hausdorff space and let $\mathbf{X}$ be an $\mathbf{E}$-completely regular space. 
	Then the following conditions are equivalent:
	\begin{enumerate}
		\item[(i)] $\mathbf{X}$ is $\mathbf{E}$-compact;
		\item[(ii)] for every Hausdorff space $\mathbf{Y}$ containing $\mathbf{X}$ as a dense subspace such that, for every $f\in C(\mathbf{X},\mathbf{E})$, there exists $\tilde{f}\in C(\mathbf{Y},\mathbf{E})$ with $\tilde{f}|_X=f$, the equality $\mathbf{X}=\mathbf{Y}$ holds;
		\item[(iii)] the set $e_{C(\mathbf{X}, \mathbf{E})}[X]$ is closed in $\mathbf{E}^{C(\mathbf{X}, \mathbf{E})}$.
	\end{enumerate}	
\end{theorem}
\begin{proof} $(i)\rightarrow (ii)$ Suppose that there exist a non-empty set $J$ and a homeomorphic embedding $h$ of $\mathbf{X}$ into  $\mathbf{E}^{J}$ such that $h[X]$ is a closed subset of $\mathbf{E}^{J}$. Let $\mathbf{Y}$ be a Hausdorff space such that $\mathbf{X}$ is a dense subspace of $\mathbf{Y}$ and, for every $f\in C(\mathbf{X},\mathbf{E})$, there exists  $\tilde{f}\in C(\mathbf{Y},\mathbf{E})$ with $\tilde{f}|_X=f$. For every $j\in J$, there exists a unique $h_j\in C(\mathbf{Y}, \mathbf{E})$ such that, for every $x\in X$, $h_j(x)=\pi_j\circ h(x)$. Let $\mathcal{H}=\{h_j: j\in J\}$. Then $e_{\mathcal{H}}\in C(\mathbf{Y}, \mathbf{E}^J)$ and $e_{\mathcal{H}}[Y]\subseteq h[X]$. Therefore, we can consider the mapping $h^{-1}\circ e_{\mathcal{H}}: Y\to X$. Since, for every $x\in X$, $h^{-1}(e_{\mathcal{H}}(x))=x$, it follows from Proposition \ref{s2:p7} that, for every $y\in Y$,  $h^{-1}(e_{\mathcal{H}}(y))=y$. This implies that $Y=X$. Hence (i) implies (ii).
	
	$(ii)\rightarrow (iii)$ Now, let $J=C(\mathbf{X}, \mathbf{E})$ and let $\mathbf{Y}$ be the subspace of $\mathbf{E}^{J}$ with the underlying set $Y=\cl_{\mathbf{E}^J}e_J[X]$. It follows from (ii) that $Y=e_J[X]$. This shows that (ii) implies (iii). 
	
	It is obvious that (iii) implies (i).
\end{proof}

We need the following theorem relevant to \cite[Theorem 3.11.16]{en}. We outline a proof of it in $\mathbf{ZF}$ because it differs from the proof of Theorem 3.11.16 given in \cite{en}. Our method is similar to that in \cite[proof of Theorem 2.1]{mr2}. One can also notice its similarity with \cite[proof of Fact 6]{nik}; however, the axiom of choice is involved in  \cite[proof of Fact 6]{nik} and, therefore, we must modify it to get a proof in $\mathbf{ZF}$. 

\begin{theorem}
	\label{s2:t9}
	$[\mathbf{ZF}]$
	Let $\mathbf{E}$ be a non-empty Hausdorff space, and $\mathbf{X}$ an $\mathbf{E}$-completely regular space. Then $e_{C(\mathbf{X}, \mathbf{E})}\mathbf{X}$ is a Hewitt $\mathbf{E}$-compact extension of $\mathbf{X}$. Furthermore, for every Hewitt $\mathbf{E}$-compact extension $\langle \mathbf{Y}, h\rangle$ of $\mathbf{X}$, it holds that $\langle\mathbf{Y}, h\rangle\approx e_{C(\mathbf{X}, \mathbf{E})}\mathbf{X}$ and, for every $\mathbf{E}$-compact space $\mathbf{Z}$ and every $f\in C(\mathbf{X},\mathbf{Z})$, there exists a unique $\tilde{f}\in C(\mathbf{Y}, \mathbf{Z})$ such that $\tilde{f}\circ h=f$.  
\end{theorem}
\begin{proof}
It is obvious that $e_{C(\mathbf{X}, \mathbf{E})}\mathbf{X}$ is a Hewitt $\mathbf{E}$-compact extension of $\mathbf{X}$. Let $\langle\mathbf{Y}, h\rangle$ be any Hewitt $\mathbf{E}$-compact extension of $\mathbf{X}$. By  Remark \ref{s2:r6} and Theorem \ref{s2:t8}, $e_{C(\mathbf{Y}, \mathbf{E})}$ is a homeomorphic embedding such that $e_{C(\mathbf{Y},\mathbf{E})}[Y]$ is closed in $\mathbf{E}^{C(\mathbf{Y}, \mathbf{E})}$. We define a homeomorphism $\psi: \mathbf{E}^{C(\mathbf{X}, \mathbf{E})}\to \mathbf{E}^{C(\mathbf{Y},\mathbf{E})}$ as follows: for every $t\in E^{C(\mathbf{X}, \mathbf{E})}$ and every $f\in C(\mathbf{Y}, \mathbf{E})$, $\psi(t)(f)= t(f\circ h)$. Let $g: e_{C(\mathbf{X}, \mathbf{E})}X\to Y$ be defined by: for every $t\in e_{C(\mathbf{X}, \mathbf{E})}X$, $g(t)=e_{C(\mathbf{Y},\mathbf{E})}^{-1}\circ \psi(t)$. Then $g$ is the required homeomorphism of $e_{C(\mathbf{X}, \mathbf{E})}\mathbf{X}$ onto $\mathbf{Y}$ showing that $e_{C(\mathbf{X}, \mathbf{E})}\mathbf{X}\approx\langle \mathbf{Y}, h\rangle$.
	
	Now, let $\mathbf{Z}$ be an $\mathbf{E}$-compact space and let $f\in C(\mathbf{X},\mathbf{Z})$. For every $j\in C(\mathbf{Z}, \mathbf{E})$, let $f_j=\pi_j\circ e_{C(\mathbf{Z}, \mathbf{E})}\circ f$  and let $\tilde{f}_j\in C(\mathbf{Y}, \mathbf{E})$ be the unique continuous extension of $f_j$ over $\mathbf{Y}$. We define $\tilde{f}: Y\to Z$ as follows: for every $y\in Y$, $\tilde{f}(y)=e_{C(\mathbf{Z},\mathbf{E})}^{-1}\circ e_{\{\tilde{f}_j: j\in C(\mathbf{Z}, \mathbf{E})\}}$ (see \cite[Lemma 3.11.15]{en}). Then $\tilde{f}\in C(\mathbf{Y},\mathbf{Z})$ and $\tilde{f}\circ h=f$.
\end{proof}

\begin{remark}
	\label{s2:r10}
	Suppose that $\mathbf{X}$ and $\mathbf{E}$ are Hausdorff spaces such that $\mathbf{X}$ is $\mathbf{E}$-completely regular. It follows from Theorem \ref{s2:t9} that $e_{C(\mathbf{X},\mathbf{E})}\mathbf{X}$ is the unique (up to the equivalence $\approx$) Hewitt $\mathbf{E}$-compact extension of $\mathbf{X}$. We denote it by $v_{\mathbf{E}}\mathbf{X}$. 
	
For a Tychonoff space $\mathbf{X}$, $v_{\mathbb{R}}\mathbf{X}$ is traditionally denoted by $v\mathbf{X}$. For a zero-dimensional $T_1$-space $\mathbf{X}$, the extension $v_{\mathbb{N}}\mathbf{X}$ is commonly denoted by $v_0\mathbf{X}$. 
\end{remark} 

\begin{remark}
\label{s2:r11}
a) Suppose that $\mathbf{X}$ is a zero-dimensional $T_1$-space. Then, for the discrete space $\mathbf{2}=\langle \{0,1\}, \mathcal{P}(\{0, 1\})\rangle$, the space $\mathbf{X}$ is $\mathbf{2}$-completely regular but $v_{\mathbf{2}}\mathbf{X}$ may fail to be compact in $\mathbf{ZF}$ (see  \cite{kw1}, \cite{ow1} and \cite{ow2}). It is true in $\mathbf{ZF}$ that if $v_{\mathbf{2}}\mathbf{X}$ is compact, then $v_{\mathbf{2}}\mathbf{X}$ is equivalent to the Banaschewski compactification $\beta_0\mathbf{X}$ of $\mathbf{X}$ (see \cite{ow1} and \cite{ow2}). It was proved in \cite{ow1} and \cite{ow2} that it holds in $\mathbf{ZF}$ that the Banaschewski compactification of $\mathbf{X}$ exists if and only if $v_{\mathbf{2}}\mathbf{X}$ is compact. 

(b) Suppose that $\mathbf{X}$ is a Tychonoff space. Then $v_{[0,1]}\mathbf{X}$ may fail to be compact in $\mathbf{ZF}$, but it is true in $\mathbf{ZF}$ that if $v_{[0,1]}\mathbf{X}$ is compact, then it is equivalent to the maximal Tychonoff compactification of $\mathbf{X}$ (see \cite{kw1}, \cite{ow1} and \cite{ow2}).

(c) $\mathbf{BPI}$ is equivalent to each of the following statements: (i) for every Tychonoff space $\mathbf{X}$, $v_{[0,1]}\mathbf{X}$ is compact, (ii) for every zero-dimensional $T_1$-space $\mathbf{X}$, $v_{\mathbf{2}}\mathbf{X}$ is compact. (Cf. \cite{kw1}, \cite{ow1} and \cite{ow2}).
\end{remark}

Fact 4 of \cite{nik} asserts that a Tychonoff space $\mathbf{X}$ is realcompact if and only if the embedding $e_{C(\mathbf{X})}$ of $\mathbf{X}$ into $\mathbb{R}^{C(\mathbf{X})}$ is closed. Unfortunately, the axiom of choice is used in the proof of Fact 4 in \cite{nik}. However, the following stronger and more general theorem  of $\mathbf{ZF}$ follows directly from our Theorem \ref{s2:t9}.

\begin{theorem}
\label{s2:t12}
$[\mathbf{ZF}]$ Let $\mathbf{E}$ be a non-empty Hausdorff space and let $\mathbf{X}$ be an $\mathbf{E}$-completely regular space. Then $\mathbf{X}$ is $\mathbf{E}$-compact if and only if the mapping $e_{C(\mathbf{X},\mathbf{E})}$ is closed.
\end{theorem}

\begin{definition}
\label{s2:d13}
Let $\mathbf{X}=\langle X, \tau\rangle$ be a locally compact, non-compact Hausdorff space, and let $\infty$ be an element such that $\infty\notin X$. We define $X(\infty):=X\cup\{\infty\}$, $\tau(\infty):=\tau\cup\{X(\infty)\setminus K: K\ \text{is compact in}\ \mathbf{X}\}$ and $\mathbf{X}(\infty):=\langle X(\infty), \tau(\infty)\rangle$. Then $\mathbf{X}(\infty)$ is the \emph{Alexandroff compactification} of $\mathbf{X}$.
\end{definition}

In much the same way, as in \cite[proof of Corollary 3.11.7]{en}, one can prove the following proposition.

\begin{proposition}
	\label{s2:p14}
	$[\mathbf{ZF}]$ Let $J$ be a non-empty set and $\mathbf{E}$ a non-empty Hausdorff space. Suppose that $\{A_j: j\in J\}$ is a family of $\mathbf{E}$-compact subspaces of a topological space $\mathbf{X}$.  Then the subspace $A=\bigcap_{j\in J}A_j$ of $\mathbf{X}$ is $\mathbf{E}$-compact.
\end{proposition}

The proof of the following proposition is similar to the proof of Corollary 3.11.8 in \cite{en}.

\begin{proposition}
	\label{s2:p15}
	$[\mathbf{ZF}]$ Let $\mathbf{E}$ be a non-empty Hausdorff space. Let $\mathbf{X}$, $\mathbf{Y}$ be Hausdorff spaces and let $f\in C(\mathbf{X}, \mathbf{Y})$. If $\mathbf{X}$ is  $\mathbf{E}$-compact, then, for every $\mathbf{E}$-compact subspace $B$ of $\mathbf{Y}$, the subspace $f^{-1}[B]$ of $\mathbf{X}$ is $\mathbf{E}$-compact. 
\end{proposition}

The following two theorems are given here for applications in the next section. We show a proof only of the second one for both can be proved by using similar arguments.

\begin{theorem}
\label{s2:t16}
$[\mathbf{ZF}]$ A zero-dimensional $T_1$-space $\mathbf{X}$ is $\mathbb{N}$-compact if and only if, for every $p\in v_{\mathbb{N}(\infty)}X\setminus X$, there exists $f\in C(v_{\mathbb{N}(\infty)}\mathbf{X}, \mathbb{N}(\infty))$ such that $f(p)=\infty$ and $f[X]\subseteq\mathbb{N}$.
\end{theorem}

\begin{theorem}
\label{s2:t17}
$[\mathbf{ZF}]$ A Tychonoff space $\mathbf{X}$ is realcompact if and only it, for every $p\in v_{\mathbb{R}(\infty)}X\setminus X$, there exists $f\in C(v_{\mathbb{R}(\infty)}\mathbf{X}, \mathbb{R}(\infty))$ such that $f(p)=\infty$ and $f[X]\subseteq\mathbb{R}$.
\end{theorem}

\begin{proof}
First, assume that $\mathbf{X}$ is a Tychonoff space. Suppose that there exists a point $p\in v_{\mathbb{R}(\infty)}X\setminus X$ such that, for every $f\in C(v_{\mathbb{R}(\infty)}\mathbf{X}, \mathbb{R}(\infty))$ with $f[X]\subseteq\mathbb{R}$, we have $f(p)\in\mathbb{R}$.
Let $Y=X\cup\{p\}$. For the subspace $\mathbf{Y}$ of $v_{\mathbb{R}(\infty)}\mathbf{X}$ and for every $f\in C(\mathbf{X})$, there exists $\tilde{f}\in C(\mathbf{Y})$ such that $\tilde{f}\upharpoonright X=f$. It follows from Theorem \ref{s2:t8} that $\mathbf{X}$ is not realcompact.

Now, suppose that, for each $p\in v_{\mathbb{R}(\infty)}X\setminus X$, there exists a function $f\in C(v_{\mathbb{R}(\infty)}\mathbf{X}, \mathbb{R}(\infty))$ such that $f(p)=\infty$ and $f[X]\subseteq\mathbb{R}$. Let
\[\mathcal{A}=\{ f\in C(v_{\mathbb{R}(\infty)}\mathbf{X}, \mathbb{R}(\infty)): f[X]\subseteq\mathbb{R}\}.\]
Since $\mathbb{R}(\infty)$ is realcompact, so is $v_{\mathbb{R}(\infty)}\mathbf{X}$. Hence, by Proposition \ref{s2:p15}, for every $f\in\mathcal{A}$, the subspace $f^{-1}[\mathbb{R}]$ of $v_{\mathbb{R}(\infty)}\mathbf{X}$ is realcompact. Since $X=\bigcap\{f^{-1}[\mathbb{R}]: f\in\mathcal{A}\}$, it follows from Proposition \ref{s2:p14} that $\mathbf{X}$ is realcompact.
\end{proof}

\section{The $\mathbf{ZF}$-proof of the Herrlich-Chew theorem}
\label{s3}

 To prove that if $\mathbf{X}$ is an $\mathbb{N}$-compact space, then every clopen ultrafilter with the countable intersection property in $\mathbf{X}$ is fixed, Chew applied in \cite[proof of Theorem A]{Chew} the Tychonoff Product Theorem to have that $v_{\mathbb{N}(\infty)}\mathbf{X}$ is compact. However, the following proposition shows that, in $\mathbf{ZF}$, the space $v_{\mathbb{N}(\infty)}\mathbf{X}$ may fail to be compact. 

\begin{proposition}
\label{s3:p1}  $[\mathbf{ZF}]$ The following statements (i) and (ii) are both equivalent to $\mathbf{BPI}$:
\begin{enumerate}
\item[(i)] For every infinite set $J$, the space $\mathbb{N}(\infty)^J$ is compact.
\item[(ii)] For every zero-dimensional $T_1$-space, $v_{\mathbb{N}(\infty)}\mathbf{X}$ is compact.
\end{enumerate}
\end{proposition} 
\begin{proof}
We use the equivalents of $\mathbf{BPI}$ cited in Remark \ref{s1:r6}(c).  Since for every infinite set $J$, the Cantor cube $\mathbf{2}^J$ is a closed subspace of the space $\mathbb{N}(\infty)^J$, and $\mathbf{BPI}$ implies that $\mathbb{N}(\infty)^J$ is compact, we deduce that (i) and $\mathbf{BPI}$ are equivalent. Clearly, (i) implies (ii).

 To show that (ii) implies (i), we fix an infinite set $J$ and put $\mathbf{X}=\mathbb{N}(\infty)$. Since $\mathbf{X}$ is $\mathbb{N}$-compact, it follows from Theorems \ref{s2:t8} and \ref{s2:t9} that $v_{\mathbb{N}(\infty)}\mathbf{X}=\mathbf{X}$. Hence (ii) implies (i).
\end{proof}

With Proposition \ref{s3:p1} in hand, one can easily check that the $\mathbf{ZFC}$-proof of  the Herrlich-Chew theorem (see Theorem \ref{s1:t3}), given in \cite{Chew}, although incorrect in $\mathbf{ZF}$, remains correct in $\mathbf{ZF+BPI}$. The main goal of this section is to modify it to a proof of Theorem \ref{s1:t7}(i) reformulated as follows:

\begin{theorem}
	\label{s3:t2}
(Cf. Theorem \ref{s1:t7}(i).) It is true in $\mathbf{ZF}$ that, for every zero-dimensional $T_1$-space $\mathbf{X}$, the space $\mathbf{X}$ is $\mathbb{N}$-compact if and only if every clopen ultrafilter with the countable intersection property in $\mathbf{X}$ is fixed. 
\end{theorem}

\begin{proof}
Let $\mathbf{X}$ be a non-empty zero-dimensional $T_1$-space. We put $J= C(\mathbf{X}, \mathbb{N}(\infty))$, $\mathbf{Y}=\mathbb{N}(\infty)^{J}$ and $P=\cl_{\mathbf{Y}}e_J[X]$. In view of Theorem  \ref{s2:t9}, for the subspace $\mathbf{P}$ of $\mathbf{Y}$, we  have $\mathbf{P}=v_{\mathbb{N}(\infty)}\mathbf{X}$. One can easily check that the following condition $(a)$ is satisfied:
\[\tag{$a$}(\forall A,B\in\mathcal{CO}(\mathbf{X})) \cl_{\mathbf{P}}(e_J[A\cap B])=\cl_{\mathbf{P}}(e_J[A])\cap\cl_{\mathbf{P}}(e_J[B]).\]

Suppose that $\mathbf{X}$ is $\mathbb{N}$-compact, and $\mathcal{U}$ is a free clopen ultrafilter in $\mathbf{X}$. We are going to prove in $\mathbf{ZF}$ that $\mathcal{U}$ does not have the countable intersection property.  To this aim, we modify the standard proof of the Tychonoff Product Theorem.  Namely, we define $\mathcal{U}^{\ast}:=\{\cl_{\mathbf{P}}(e_J[U]): U\in\mathcal{U}\}$ and, for every $j\in J$, we put
 \[\mathcal{U}_j:=\{\cl_{\mathbb{N}(\infty)}\pi_j[U]: U\in\mathcal{U}^{\ast}\}.\] 
\noindent We notice that it follows from $(a)$ that $\mathcal{U}^{\ast}$ is a clopen ultrafilter in $\mathbf{P}$. Since $\mathcal{U}^{\ast}$ is a filter, it follows that, for every $j\in J$, the family $\mathcal{U}_j$ has the finite intersection property. By the compactness of $\mathbb{N}(\infty)$, for every $j\in J$, $\bigcap\mathcal{U}_J\neq\emptyset$. Therefore, since $\mathbb{N}(\infty)$ is well-orderable, we can fix $p\in Y$ such that, for every $j\in J$, $p(j)\in\bigcap\mathcal{U}_j$. It is easy to check that $p\in P$. Suppose that there exists $U_0\in\mathcal{U}^{\ast}$ such that $p\notin U_0$. There exists a family $\{V_j: j\in J\}$ of clopen sets of $\mathbb{N}(\infty)$ such that the set $K=\{j\in J: V_j\neq \mathbb{N}(\infty)\}$ is finite and, for $V=\prod_{j\in J}V_j$, we have $p\in V\subseteq Y\setminus U_0$. For every $j\in K$ and $U\in\mathcal{U}^{\ast}$, we have $U\cap \pi_j^{-1}[V_j]\neq\emptyset$, so $\pi_j^{-1}[V_j]\cap P\in\mathcal{U}^{\ast}$ because $\mathcal{U}^{\ast}$ is a clopen ultrafilter in $\mathbf{P}$. This implies that $V\cap P\in\mathcal{U}^{\ast}$ which is impossible. The contradiction obtained proves that $p\in\bigcap\mathcal{U}^{\ast}$. Since $\mathcal{U}$ is free, we deduce that $p\notin e_J[X]$. This, together with the $\mathbb{N}$-compactness of $\mathbf{X}$ and Theorem \ref{s2:t16}, implies that there exists $f\in C(\mathbf{P}, \mathbb{N}(\infty))$ such that $f(p)=\infty$ and $f[e_J[X]]\subseteq \mathbb{N}$. In much the same way, as in the proof of Theorem A in \cite{Chew}, for every $n\in\mathbb{N}$, we consider the set $V_n=\mathbb{N}(\infty)\setminus\{i\in\mathbb{N}: i\leq n\}$ and $W_n=f^{-1}[V_n]$. We observe that, for every $n\in\mathbb{N}$,  $e_J^{-1}[W_n]\in\mathcal{U}$ but $\bigcap_{n\in\mathbb{N}}e_J^{-1}[W_n]=\emptyset$, so $\mathcal{U}$ does not have the countable intersection property. 
	
	Now, suppose that $\mathbf{X}$ is not $\mathbb{N}$-compact. It follows from Theorem \ref{s2:t16} that there exists $p_0\in P\setminus e_J[X]$ such that, for every $f\in C(\mathbf{P}, \mathbb{N}(\infty))$ with $f[e_J[X]]\subseteq\mathbb{N}$, we have $f(p_0)\in\mathbb{N}$. Let $\mathcal{F}=\{ A\in\mathcal{CO}(\mathbf{X}): p_0\in\cl_{\mathbf{P}}(e_J[A])\}$. It follows from $(a)$ that $\mathcal{F}$ is a clopen ultrafilter in $\mathbf{X}$ such that $\bigcap\mathcal{F}=\emptyset$.  To show that $\mathcal{F}$ has the countable intersection property, we fix a subfamily $\{F_n: n\in\mathbb{N}\}$ of $\mathcal{F}$ such that, for every $n\in\mathbb{N}$, $F_{n+1}\subseteq F_n$. Suppose that $\bigcap_{n\in\mathbb{N}}F_n=\emptyset$. In much the same way, as in \cite[proof of Theorem A]{Chew}, we define a function $f_0\in C(\mathbf{X}, \mathbb{N})$ as follows:
	\[ 
	f_0(x)=\begin{cases} 1&\quad\text{if}\quad x\in X\setminus F_1,\\
	n+1&\quad\text{if}\quad x\in F_{n}\setminus F_{n+1}.\end{cases} \]
	There exists a unique $g_0\in C(\mathbf{P}, \mathbb{N}(\infty))$ such that $g_0\circ e_J=f_0$. Then $g_0(p_0)=\infty$, and this is impossible. The contradiction obtained completes the proof.
\end{proof}

One can notice that Chew's proof, given in \cite{Chew}, that if every clopen ultrafilter with the countable intersection property in a zero-dimensional $T_1$-space $\mathbf{X}$ is fixed, then $\mathbf{X}$ is $\mathbb{N}$-compact, is correct also in $\mathbf{ZF}$, but we have clarified it.

\section{A $c_{\delta}$-modification of the Herrlich-Chew theorem}
\label{s4}

Since it is interesting in itself if ``clopen ultrafilter'' can be replaced with  ``$c_{\delta}$-ultrafilter'' in Theorem \ref{s3:t2}, we want to prove Theorem \ref{s1:t7}(iv) in this section. To do this, we need the following lemma.

\begin{lemma}
	\label{s4:l1}
	$[\mathbf{ZF}]$
	Let $\mathcal{A}$ be a non-empty family of subsets of a set $X$ such that $\mathcal{A}$ is closed under finite unions and finite intersections. Then $\mathbf{CMC}$ implies that $\mathcal{A}_{\delta\delta}=\mathcal{A}_{\delta}$. 
\end{lemma}
\begin{proof}
	Assuming $\mathbf{CMC}$, it suffices to show that $\mathcal{A}_{\delta\delta}\subseteq\mathcal{A}_{\delta}$. Suppose that, for every $n\in\omega$, $A_n\in\mathcal{A}_{\delta}$ and $A=\bigcap_{n\in\omega}A_n$. For every $n\in\omega$, let 
	\[\mathcal{B}_n=\{ \mathcal{B}\in [\mathcal{A}]^{\leq\omega}\setminus\{\emptyset\}: A_n=\bigcap\mathcal{B}\}.\]
	\noindent By $\mathbf{CMC}$, there exists a family $\{\mathcal{F}_n: n\in\omega\}$ such that, for every $n\in\omega$, $\mathcal{F}_n$ is a non-empty finite subfamily of $\mathcal{B}_n$. For every $n\in\omega$, let $\mathcal{U}_n=\bigcup\mathcal{F}_n$. Then, for every $n\in\omega$, $\mathcal{U}_n\in [\mathcal{A}]^{\leq\omega}\setminus\{\emptyset\}$ and $A_n=\bigcap\mathcal{U}_n$. Since $\mathbf{CMC}$ implies $\mathbf{A}(\delta\delta)$ (see Remark \ref{s1:r6}(a)), there exists a family $\{V_{m,n}: m,n\in\omega\}$ of members of $\mathcal{A}$ such that, for every $n\in\omega$, $A_n=\bigcap_{m\in\omega}V_{m,n}$. Then $A=\bigcap_{n\in\omega}A_n\in \mathcal{A}_{\delta}$.
\end{proof}

\begin{theorem}
	\label{s4:t2} 
	(Cf. Theorem \ref{s1:t7}(iii).) $[\mathbf{ZF}]$ Let $\mathbf{X}$ be a zero-dimensional $T_1$-space. Then the following conditions are satisfied:
	\begin{enumerate}
		\item[(i)] if $\mathbf{X}$ is  $\mathbb{N}$-compact, then every $c_{\delta}$-ultrafilter with the countable intersection property in $\mathbf{X}$ is fixed;
		\item[(ii)] $\mathbf{CMC}$ implies that if every $c_{\delta}$- ultrafilter  with the countable intersection property in $\mathbf{X}$ is fixed, then $\mathbf{X}$ is $\mathbb{N}$-compact.
	\end{enumerate}
\end{theorem}
\begin{proof}
	(i) Suppose that the space $\mathbf{X}$ is  $\mathbb{N}$-compact, and $\mathcal{U}$ is a $c_{\delta}$-ultrafilter in $\mathbf{X}$ such that $\mathcal{U}$ has the countable intersection property. We define
	\[\mathcal{U}^{\ast}:=\{V\in\mathcal{CO}(\mathbf{X}): (\exists U\in\mathcal{U}) U\subseteq V\}.\]
	\noindent One can easily check that $\mathcal{U}^{\ast}$ is a clopen ultrafilter in $\mathbf{X}$ such that $\mathcal{U}^{\ast}\subseteq\mathcal{U}$. It follows from the countable intersection property of $\mathcal{U}$ that $\mathcal{U}^{\ast}$ has the countable intersection property. Since $\mathbf{X}$ is $\mathbb{N}$-compact, we deduce from Theorem \ref{s3:t2} that there exists $p\in\bigcap\mathcal{U}^{\ast}$. Let $U\in\mathcal{U}$. There exists a family $\{U_n: n\in\omega\}$ of clopen sets in $\mathbf{X}$ such that $U=\bigcap_{n\in\omega}U_n$. For every $n\in\omega$, we have $U_n\in\mathcal{U}^{\ast}$ and, in consequence, $p\in U_n$. Hence $p\in U$. This shows that $p\in\bigcap\mathcal{U}$.\medskip
	
	(ii) Now, suppose that $\mathbf{X}$ is not $\mathbb{N}$-compact. By Theorem \ref{s3:t2}, there exists a free clopen ultrafilter $\mathcal{U}$ in $\mathbf{X}$ such that $\mathcal{U}$ has the countable intersection property. We define
	$$\mathcal{V}=\{ V\in\mathcal{CO}_{\delta}(\mathbf{X}): (\exists U\in \mathcal{U}_{\delta}) U\subseteq V\}.$$
	It is easily seen that $\mathcal{V}$ is a filter in $\mathcal{CO}_{\delta}(\mathbf{X})$ such that $\mathcal{U}\subseteq \mathcal{V}$. To check that $\mathcal{V}$ is an ultrafilter in $\mathcal{CO}_{\delta}(\mathbf{X})$, we consider any $A\in\mathcal{CO}_{\delta}(\mathbf{X})$ such that $A\notin\mathcal{V}$. We choose $\mathcal{A}\in[\mathcal{CO}(\mathbf{X})]^{\leq\omega}\setminus\{\emptyset\}$ such that $A=\bigcap\mathcal{A}$. Since $A\notin\mathcal{V}$, there exists $A_0\in\mathcal{A}$ such that $A_0\notin\mathcal{U}$. Since $\mathcal{U}$ is a clopen ultrafilter, there exists $U_0\in\mathcal{U}$ such that $A_0\cap U_0=\emptyset$. Then $A\cap U_0=\emptyset$. This shows that $\mathcal{V}$ is a $c_{\delta}$-ultrafilter in $\mathbf{X}$. Since $\mathcal{U}$ is free, so is $\mathcal{V}$.
	
	Assuming $\mathbf{CMC}$, let us show that $\mathcal{V}$ has the countable intersection property. To this aim, assume that, for every $n\in\omega$, $V_n\in\mathcal{V}$ and 
	\[\mathcal{A}_n=\{U\in\mathcal{U}_{\delta}: U\subseteq V_n\}.\]
	\noindent Since $\mathcal{U}$ is an ultrafilter in $\mathcal{CO}(\mathbf{X})$, $\mathcal{U}$ is closed under finite unions and finite intersections. Therefore, in view of Lemma \ref{s4:l1}, it follows from $\mathbf{CMC}$ that $\mathcal{U}_{\delta\delta}=\mathcal{U}_{\delta}$.  By $\mathbf{CMC}$, there exists a sequence $(\mathcal{F}_n)_{n\in\omega}$ such that, for every $n\in\omega$, $\mathcal{F}_n$ is a non-empty finite subset of $\mathcal{A}_n$. For every $n\in\omega$, let $U_n=\bigcap \mathcal{F}_n$. Then, for every $n\in\omega$, $U_n\in\mathcal{U}_{\delta}$, so $W=\bigcap_{n\in\omega}U_n\in\mathcal{U}_{\delta\delta}$. This, together with $\mathbf{CMC}$, implies that $W\in\mathcal{U}_{\delta}$. Hence $W\neq\emptyset$ because $\mathcal{U}$ has the countable intersection property.  Since $W\subseteq\bigcap_{n\in\omega}V_n$, we have $\bigcap_{n\in\omega}V_n\neq\emptyset$ . Hence, $\mathbf{CMC}$ implies that $\mathcal{V}$ is a free $c_{\delta}$-ultrafilter in $\mathbf{X}$ with the countable intersection property. 
\end{proof}

\section{A modification of Hewitt's theorem}
\label{s5}

Our aim is to prove items (ii) and (iii) of Theorem \ref{s1:t7} by modifying the proof of Theorem \ref{s3:t2} in this section. We begin with two lemmas.

\begin{lemma}
	\label{s5:l1}
	$[\mathbf{ZF}]$ For a topological space $\mathbf{X}$, let $\mathcal{U}\subseteq\mathcal{Z}(\mathbf{X})$. Then $\mathbf{CMC}$ implies that the following conditions are satisfied:
	\begin{enumerate}
	\item[(i)] every countable subfamily of $\mathcal{U}$ is functionally accessible (see Definition \ref{s1:d1});
	\item[(ii)] $\mathcal{U}$ has the weak countable intersection property if and only if it has the countable intersection property;
	\item[(iii)] the families $\mathcal{Z}(\mathbf{X})$ and $\mathcal{CO}_{\delta}(\mathbf{X})$ are both closed under countable intersections.
	\end{enumerate}
\end{lemma}

\begin{proof} That (i) and (iii) hold in $\mathbf{ZF+CMC}$, was shown in \cite[proof of Theorem 4.9(i)]{kow}. That (ii) holds in $\mathbf{ZF+CMC}$, follows directly from (i) and Definition \ref{s1:d1}. One can also deduce from Lemma \ref{s4:l1} that $\mathcal{CO}_{\delta}(\mathbf{X})$ is closed under countable intersections in $\mathbf{ZF+CMC}$.
\end{proof}

We omit a simple standard proof of the following lemma.

\begin{lemma}
\label{s5:l2}
$[\mathbf{ZF}]$ 
Let $\mathbf{X}$ be a Tychonoff space, $\mathbf{Y}=v_{\mathbb{R}(\infty)}\mathbf{X}$ and let $Z_1, Z_2\in\mathcal{Z}(\mathbf{X})$. Then 
\[ \cl_{\mathbf{Y}}(Z_1)\cap\cl_{\mathbf{Y}}(Z_2)=\cl_{\mathbf{Y}}(Z_1\cap Z_2).\]
\end{lemma}

For our convenience, we slightly reformulate Theorem \ref{s1:t7}(ii) as follows. 

\begin{theorem}
	\label{s5:t3} (Theorem \ref{s1:t7}(ii).)
	$[\mathbf{ZF}]$ For every Tychonoff space $\mathbf{X}$, the following conditions are equivalent:
	\begin{enumerate}
	\item[(i)] $\mathbf{X}$ is realcompact;
	\item[(ii)]  every $z$-ultrafilter in $\mathbf{X}$ with the weak countable intersection property is fixed.
	\end{enumerate} 
\end{theorem}

\begin{proof}
	For a Tychonoff space $\mathbf{X}$, let $J=C(\mathbf{X}, \mathbb{R}(\infty))$, $\mathbf{Y}=\mathbb{R}(\infty)^J$ and let $P$ be the closure in $\mathbf{Y}$ of $e_J[X]$. Then, by Theorems \ref{s2:t8} and \ref{s2:t9}, $\mathbf{P}=v_{\mathbb{R}(\infty)}\mathbf{X}$.
	
	$(i)\to (ii)$ Assume that $\mathbf{X}$ is realcompact, and $\mathcal{U}$ is a free $z$-ultrafilter in $\mathbf{X}$. For every $j\in J$, let $\mathcal{A}_j=\{\cl_{\mathbb{R}(\infty)}\pi_j[e_J[U]]: U\in\mathcal{U}\}$ and $A_j=\bigcap\mathcal{A}_j$. Consider any $j\in J$. It follows from Lemma \ref{s5:l2} that the family $\mathcal{A}_j$ is centered. By the compactness of $\mathbb{R}(\infty)$, $A_j\neq\emptyset$. If $\infty\in A_j$, we put $p(j)=\infty$.  Suppose that $\infty\notin A_j$. Then $A_j$ is a non-empty compact subset of $\mathbb{R}$, so we can define $p(j)=\max(A_j)$. In this way, we have defined in $\mathbf{ZF}$ a point $p\in\mathbb{R}(\infty)^J$ such that, for every $j\in J$, $p(j)\in A_j$. Suppose that there exists $U_0\in\mathcal{U}$ such that $p\notin\cl_{\mathbf{Y}}(e_J[U_0])$. There exists a family $\{V_j: j\in J\}$ of zero-sets of $\mathbb{R}(\infty)$ such that the set $K=\{j\in J: V_j\neq \mathbb{R}(\infty)\}$ is finite and, for $V=\prod_{j\in J}V_j$, we have $p\in\inter_{\mathbf{Y}}(V)$ and $V\cap\cl_{\mathbf{Y}}(e_J[U_0])=\emptyset$. If $j\in K$, then $e_J^{-1}[\pi_j^{-1}[V_j]]\in\mathcal{Z}(\mathbf{X})$ and, for every $U\in\mathcal{U}$, $U\cap e_J^{-1}[\pi_j^{-1}[V_j]]\neq\emptyset$. Therefore, since  $\mathcal{U}$ is a $z$-ultrafilter, we deduce that, for every $j\in K$,  $e_J^{-1}[\pi_j^{-1}[V_j]]\in\mathcal{U}$. This implies that $e_J^{-1}[V]\in\mathcal{U}$. This is impossible because $e_J^{-1}[V]\cap U_0=\emptyset$. The contradiction obtained proves that $p\in\bigcap\{\cl_{\mathbf{Y}}e_J[U]: U\in\mathcal{U}\}$. Since $\mathcal{U}$ is free and $p\in P$, we deduce that $p\in P\setminus e_J[X]$.  It follows from the realcompactness of $\mathbf{X}$ and from Theorem \ref{s2:t17} that there exists $f\in C(\mathbf{P}, \mathbb{R}(\infty))$ such that $f[e_J{X}]\subseteq\mathbb{R}$ and $f(p)=\infty$. For every $n\in\mathbb{N}$, let $Z_n=f^{-1}[\mathbb{R}(\infty)\setminus (-n, n)]$ and $W_n=e_J^{-1}[Z_n]$. Then $Z_n$ are zero-sets in $\mathbf{P}$ such that for every $U\in\mathcal{U}$, $Z_n\cap e_J[U]\neq\emptyset$. This implies that, for every $n\in\mathbb{N}$, $W_n\in\mathcal{U}$. The family $\{W_n: n\in\mathbb{N}\}$ of zero-sets of $\mathbf{X}$ is functionally accessible. However, $\bigcap_{n\in\mathbb{N}}W_n=\emptyset$ because $\bigcap_{n\in\mathbb{N}}Z_n\cap e_J[X]=\emptyset$. Hence $\mathcal{U}$ does not have the weak countable intersection property.
	
	$(ii)\to (i)$ Now, let us assume (ii) and consider any point $p\in P\setminus e_J[X]$.
We define $\mathcal{F}: =\{Z\in\mathcal{Z}(\mathbf{X}): p\in\cl_{\mathbf{Y}}e_J[Z]\}$. Since $\mathbf{P}$ is Tychonoff, one can easily check that Lemma \ref{s5:l2} implies that $\mathcal{F}$ is a free  $z$-ultrafilter in $\mathbf{X}$. By (ii), $\mathcal{F}$ does not have the weak countable intersection property. Hence, there exists a family $\{f_n: n\in\mathbb{N}\}$ of functions from $C(\mathbf{X}, [0, 1])$ such that, for every $n\in\mathbb{N}$, $Z(f_n)\in\mathcal{F}$ but $\bigcap_{n\in\mathbb{N}}Z(f_n)=\emptyset$. We define the function $f\in C(\mathbf{X}, [0,1])$ by: $f(x)=\sum_{n=1}^{+\infty}\frac{f_n(x)}{2^n}$ for every $x\in X$. Since, for every $x\in X$, $f(x)\neq 0$, we have the function $\frac{1}{f}\in C(\mathbf{X}, \mathbb{R})$. Then $\frac{1}{f}\in C(\mathbf{X}, \mathbb{R}(\infty))$. There exists $g\in C(\mathbf{P}, \mathbb{R}(\infty))$ such that $g\circ e_J=\frac{1}{f}$. Then $g(p)=\infty$ and $g[e_J[X]]\subseteq\mathbb{R}$. Hence $\mathbf{X}$ is realcompact by Theorem \ref{s2:t17}.
\end{proof}

The following corollary is an immediate consequence of Theorem \ref{s5:t3}, Lemma \ref{s5:l1} and the fact that if $\mathbf{X}$ is a metrizable space, then every non-empty family of zero-sets of $\mathbf{X}$ is functionally accessible,

\begin{corollary}
	\label{s5:c4}
(Cf. Theorem \ref{s1:t7}(iii).)	$[\mathbf{ZF}]$ A metrizable space $\mathbf{X}$ is realcompact if and only if every $z$-ultrafilter in $\mathbf{X}$ with the countable intersection property is fixed. Furthermore, $\mathbf{CMC}$ implies that a Tychonoff space $\mathbf{X}$ is realcompact if and only if every $z$-ultrafilter in $\mathbf{X}$ with the countable intersection property is fixed.
\end{corollary}

\section{Several applications of the main results}
\label{s6}

It is well known that, in $\mathbf{ZFC}$, all regular Lindel\"of $T_1$-spaces are realcompact (see, e.g., \cite[Theorem 52]{hewitt},  \cite[Corollary 14]{nik}, \cite[Theorem 3.11.12]{en} and \cite[Theorem 8.2]{gj}). On the other hand, there are models of $\mathbf{ZF}$ in which some compact Hausdorff spaces are not realcompact because they are not Tychonoff (see, e.g., \cite{lauch}, \cite{gt} and \cite[Forms 78 and 155]{hr}). We do not know if there is a model of $\mathbf{ZF}$ in which a Tychonoff Lindel\"of space need not be realcompact (see Problem \ref{s7:4}). However, we can offer the following theorem.

\begin{theorem}
\label{s6:t2}
$[\mathbf{ZF}]$ 
\begin{enumerate}
\item[(i)] Every zero-dimensional Lindel\"of $T_1$-space is $\mathbb{N}$-compact, so also realcompact.
\item[(ii)] $\mathbf{CMC}$ implies that every Tychonoff Lindel\"of space is realcompact.
\end{enumerate}
\end{theorem}
\begin{proof}
We notice that if $\mathcal{U}$ be a non-empty family of closed sets of a Lindel\"of space such that  $\bigcap\mathcal{U}=\emptyset$, then $\mathcal{U}$ does not have the countable intersection property. Hence, (i) follows from Theorem \ref{s3:t2}, and (ii) is a consequence of Corollary \ref{s5:c4}.
\end{proof}

It has been known for a long time that every strongly zero-dimensional realcompact space is $\mathbb{N}$-compact in $\mathbf{ZFC}$ (see  \cite{mr1} and \cite{mr4}). To show that this result also hold in $\mathbf{ZF+CMC}$, we need the following lemma.

\begin{lemma}
\label{s6:l4}
$[\mathbf{ZF}]$ Let $\mathbf{X}$ be a strongly zero-dimensional realcompact space and let $\{Z_n: n\in\omega\}\subseteq\mathcal{Z}(\mathbf{X})$. Then $\mathbf{CMC}$ implies that there exists a family $\{C_n: n\in\omega\}$ of clopen sets in $\mathbf{X}$ such that $\bigcap_{n\in\omega}Z_n=\bigcap_{n\in\omega}C_n$ and, for every $m\in\omega$, $\bigcap_{i\in m+1}Z_i\subseteq C_m$.
\end{lemma}
\begin{proof}
Assuming $\mathbf{CMC}$, it follows from Lemma \ref{s5:l1} that there exists a family $\{f_n: n\in\omega\}$ such that, for every $n\in\omega$, $f_n\in C(\mathbf{X}, [0,1])$ and $Z_n=Z(f_n)$. For every $n\in\omega$, let  $g_n=\sum_{i\in n+1}f_i$. Then, for every $n\in\omega$, $\bigcap_{i\in n+1}Z_i=Z(g_n)$ and $g_n\leq g_{n+1}$. For every $n\in\omega$, let
 \[\mathcal{D}_n=\Bigl\{D\in\mathcal{CO}(\mathbf{X}): Z(g_n)\subseteq D\subseteq \Bigl\{x\in X: \frac{1}{2}\leq g_n(x)\Bigr\}\Bigr\}.\]
 \noindent Since $\mathbf{X}$ is strongly zero-dimensional, the families $\mathcal{D}_n$ are all non-empty. It follows from $\mathbf{CMC}$ that there exists a family $\{\mathcal{C}_n: n\in\omega\}$ of non-empty finite sets such that, for every $n\in\omega$, $\mathcal{C}_n\subseteq\mathcal{D}_n$. For every $n\in\omega$, let $C_n=\bigcup\mathcal{C}_n$. Then the sets $C_n$ are all clopen in $\mathbf{X}$ and, moreover, $Z(g_n)\subseteq C_n\subseteq\bigl\{x\in X: \frac{1}{2}\leq g_n(x)\bigr\}$. Then $\bigcap_{n\in\omega}Z_n\subseteq \bigcap_{n\in\omega}C_n$. Suppose that $x\in X\setminus\bigcap_{n\in\omega}Z_n$. Then there exists $m_0\in\omega$ such that $x\notin Z(g_{m_0})$. There exists $m_1\in\omega$ such that $m_0<m_1$ and $\frac{1}{2^{m_1}}<g_{m_0}(x)$. Then $\frac{1}{2^{m_1}}<g_{m_1}(x)$, so $x\notin C_{m_1}$. This shows that $\bigcap_{n\in\omega}C_n\subseteq\bigcap_{n\in\omega}Z_n$. Hence $\bigcap_{n\in\omega}C_n=\bigcap_{n\in\omega}Z_n$.
\end{proof}

\begin{theorem}
\label{s6:t5}
$[\mathbf{ZF+CMC}]$ Every strongly zero-dimensional realcompact space is $\mathbb{N}$-compact.
\end{theorem}
\begin{proof}
Let $\mathbf{X}$ be a strongly zero-dimensional realcompact space. Consider any clopen ultrafilter $\mathcal{U}$ in $\mathbf{X}$. Let 
\[\mathcal{F}=\{Z\in\mathcal{Z}(\mathbf{X}): (\forall C\in\mathcal{U}) C\cap Z\neq\emptyset\}.\]
Of course, $\mathcal{U}\subseteq\mathcal{F}$. Let us check that $\mathcal{F}$ is a $z$-ultrafilter in $\mathbf{X}$.

Using the strong zero-dimensionality of $\mathbf{X}$, one can easily check that $\mathcal{F}$ is a $z$-ultrafilter in $\mathbf{X}$.
Let us assume that $\mathcal{U}$ has the countable intersection property and show that $\mathcal{F}$ also has the countable intersection property. To this aim, consider any family $\{Z_n: n\in\omega\}$ of members of $\mathcal{F}$. By $\mathbf{CMC}$ and Lemma \ref{s6:l4}, there exists a family $\{C_n: n\in\omega\}$ of clopen sets of $\mathbf{X}$ such that $\bigcap_{n\in\omega}Z_n=\bigcap_{n\in\omega}C_n$ and, for every $m\in\omega$, $\bigcap_{n\in m+1}Z_n\subseteq C_m$. Then, for every $m\in\omega$, $C_m\in\mathcal{F}$. In consequence, since $\mathcal{U}$ is a clopen ultrafilter in $\mathbf{X}$, we deduce that, for every $m\in\omega$, $C_m\in\mathcal{U}$.  Since $\mathcal{U}$ has the countable intersection property, we have $\bigcap_{n\in\omega}C_n\neq\emptyset$. Hence $\bigcap_{n\in\omega} Z_n\neq\emptyset$. Since $\mathbf{X}$ is realcompact, it follows from Corollary \ref{s5:c4} that $\bigcap\mathcal{F}\neq\emptyset$. Since $\mathcal{U}\subseteq\mathcal{F}$, we infer that $\bigcap\mathcal{U}\neq\emptyset$. This proves that $\mathbf{X}$ is $\mathbb{N}$-compact by Theorem \ref{s3:t2}.
\end{proof}

We recall that a subspace $\mathbf{Y}$ of a topological space $\mathbf{X}$ is $z$-\emph{embedded} in $\mathbf{X}$ if, for every $Z\in\mathcal{Z}(\mathbf{Y})$, there exists $\tilde{Z}\in\mathcal{Z}(\mathbf{X})$ such that $Z=\tilde{Z}\cap Y$. In \cite[Lemma 3.9]{Bl}, it was stated in $\mathbf{ZFC}$ that if a Tychonoff space $\mathbf{X}$ is a countable union of $z$-embedded realcompact subspaces, then $\mathbf{X}$ is realcompact. By proving Theorem \ref{s6:t6} below, we show that Lemma 3.9 of \cite{Bl} is provable in $\mathbf{ZF+CMC}$. In our proof, we clearly indicate where $\mathbf{CMC}$ is applied.

\begin{theorem}
	\label{s6:t6}
	$[\mathbf{ZF}]$ $\mathbf{CMC}$ implies that if a Tychonoff space $\mathbf{X}$ is a countable union of $z$-embedded realcompact subspaces, then $\mathbf{X}$ is realcompact.
\end{theorem}

\begin{proof}
In what follows, we assume $\mathbf{CMC}$. Let $\mathbf{X}$ be a Tychonoff space such that $X=\bigcup_{n\in\mathbb{N}}X_n$ where, for every $n\in\mathbb{N}$, the subspace $\mathbf{X}_n$ of $\mathbf{X}$ is realcompact and $z$-embedded in $\mathbf{X}$. Let $\mathcal{U}$ be a $z$-ultrafilter in $\mathbf{X}$ such that $\mathcal{U}$ has the countable intersection property. Suppose that, for every $n\in\mathbb{N}$, the set $\mathcal{U}_n=\{U\in\mathcal{U}: U\cap X_n=\emptyset\}$ is non-empty. By $\mathbf{CMC}$, there exists $\psi\in\prod_{n\in\mathbb{N}}([\mathcal{U}_n]^{<\omega}\setminus\{\emptyset\})$. Then, for every $n\in\mathbb{N}$, the set $U_n=\bigcap\psi(n)$ is a member of $\mathcal{U}$ such that $U_n\cap X_n=\emptyset$. Since $\bigcap_{n\in\mathbb{N}}U_n=\emptyset$, we obtain a contradiction with the countable intersection property of $\mathcal{U}$. Therefore, there exists $n_0\in\mathbb{N}$ such that $\mathcal{U}_{n_0}=\emptyset$. Let $\mathcal{F}=\{U\cap X_{n_0}: U\in\mathcal{U}\}$. Then $\mathcal{F}$ is a $z$-filter in $\mathbf{X}_{n_0}$. Since $\mathbf{X}_{n_0}$ is $z$-embedded in $\mathbf{X}$, $\mathcal{F}$ is a $z$-ultrafilter in $\mathbf{X}_{n_0}$. For every $n\in\mathbb{N}$, let $F_n\in\mathcal{F}$ and $\mathcal{W}_n=\{U\in\mathcal{U}: F_n=U\cap X_{n_0}\}$. Since $\mathcal{U}$ is closed under finite intersections, in much the same way,  as in the proof that, for some $n\in\mathbb{N}$,  $\mathcal{U}_n=\emptyset$, we can deduce from $\mathbf{CMC}$ that there exists a family $\{A_n: n\in\mathbb{N}\}$ of members of $\mathcal{U}$ such that, for every $n\in\mathbb{N}$, $F_n=A_n\cap X_{n_0}$. Let $A=\bigcap_{n\in\mathbb{N}}A_n$. It follows from $\mathbf{CMC}$ and Lemma \ref{s5:l1}(iii) that $A\in\mathcal{Z}(\mathbf{X})$. Since the $z$-ultrafilter $\mathcal{U}$ has the countable intersection property, we infer that $A\in\mathcal{U}$. Hence $A\cap X_{n_0}\neq\emptyset$, so $\bigcap_{n\in\mathbb{N}}F_n\neq\emptyset$. Therefore, $\mathcal{F}$ has the countable intersection property. Since $\mathbf{X}_{n_0}$ is realcompact, the $z$-ultrafilter $\mathcal{F}$ is fixed by Corollary \ref{s5:c4}. Hence $\mathcal{U}$ is also fixed. In the light of Theorem \ref{s5:t3}, $\mathbf{X}$ is realcompact.
\end{proof}

To establish an accurate modification of Theorem  \ref{s6:t6} for $\mathbb{N}$-compactness, we need the following definition.

\begin{definition}
	\label{s6:d7}
	A subspace $\mathbf{S}$ of a topological space $\mathbf{X}$ will be called:
	\begin{enumerate}
		\item[(i)] $\mathbf{2}$-\emph{embedded} in $\mathbf{X}$ if, for every $A\in\mathcal{CO}(\mathbf{S})$, there exists $B\in\mathcal{CO}(\mathbf{X})$ such that $A=B\cap S$;
		\item[(ii)] $c_{\delta}$-\emph{embedded} in $\mathbf{X}$ if, for every $A\in\mathcal{CO}_{\delta}(\mathbf{S})$, there exists $B\in\mathcal{CO}_{\delta}(\mathbf{X})$ such that $A=B\cap S$.
	\end{enumerate}
\end{definition}

The concept of $\mathbf{2}$-embeddability is not new. That a subspace $\mathbf{S}$ of a topological space $\mathbf{X}$ is $\mathbf{2}$-embedded in $\mathbf{X}$ is equivalent to the condition asserting that, for every $f\in C(\mathbf{S},\mathbf{2})$, there exists $\tilde{f}\in C(\mathbf{X}, \mathbf{2})$ with $\tilde{f}\upharpoonright S=f$.

\begin{proposition}
	\label{s6:p8}
	$[\mathbf{ZF}]$  $\mathbf{CMC}$ implies that, for every topological space $\mathbf{X}$, the following conditions are satisfied:
	\begin{enumerate}
		\item[(i)]  for every $c_{\delta}$-ultrafilter in $\mathbf{X}$, it holds that $\mathcal{U}$ has the countable intersection property if and only if $\mathcal{U}$ is closed under countable intersections; 
		\item[(ii)]  for every subspace $\mathbf{S}$ of $\mathbf{X}$, it holds that if $\mathbf{S}$ is $\mathbf{2}$-embedded in $\mathbf{X}$, then it is also $c_{\delta}$-embedded in $\mathbf{X}$.
	\end{enumerate}
\end{proposition} 

\begin{proof}
	In what follows, we assume $\mathbf{CMC}$. Let $\mathbf{X}$ be a given topological space.
	
	(i) Suppose that $\mathcal{U}$ is a $c_{\delta}$-ultrafilter in $\mathbf{X}$. Let $\{F_n: n\in\omega\}$ be a family of members of $\mathcal{U}$ and  let $F=\bigcap_{n\in\omega}F_n$. By Lemma \ref{s5:l1}(iii), $\mathbf{CMC}$ implies that $\mathcal{CO}_{\delta}(\mathbf{X})$ is closed under countable intersections. Therefore, $F\in\mathcal{CO}_{\delta}(\mathbf{X})$. Suppose that  $F\notin\mathcal{U}$. Since $F\in\mathcal{CO}_{\delta}(\mathbf{X})$, and $\mathcal{U}$ is a $c_{\delta}$-ultrafilter in $\mathbf{X}$, we can fix $U_F\in\mathcal{U}$ such that $F\cap U_F=\emptyset$. Then $\{ U_F\}\cup\{F_n: n\in\omega\}$ witnesses that $\mathcal{U}$ does not have the countable intersection property. On the other hand, if a $c_{\delta}$-ultrafilter in $\mathbf{X}$ is closed under countable intersections, then it has the countable intersection property. This completes the proof of (i).
	
	(ii) Suppose that $\mathbf{S}$ is a $\mathbf{2}$-embedded subspace of $\mathbf{X}$. Let $C\in\mathcal{CO}_{\delta}(\mathbf{S})$. Let $\{C_n: n\in\omega\}$ be a family of members of $\mathcal{CO}(\mathbf{S})$ such that $C=\bigcap_{n\in\omega}C_n$. Since $\mathbf{S}$ is $\mathbf{2}$-embedded in $\mathbf{X}$, for every $n\in\omega$, the family 
	\[\mathcal{D}_n=\{D\in\mathcal{CO}(\mathbf{X}): C_n=D\cap S\}\]
	\noindent is non-empty. By $\mathbf{CMC}$, there exists $\psi\in\prod_{n\in\omega}([\mathcal{D}_n]^{<\omega}\setminus\{\emptyset\})$. For every $n\in\omega$, let $D_n=\bigcap\psi(n)$. Let $D=\bigcap_{n\in\omega}D_n$. Since, for every $n\in\omega$, $D_n\in\mathcal{CO}(\mathbf{X})$, we have $D\in\mathcal{CO}_{\delta}(\mathbf{X})$. It is obvious that $D\cap S=C$, so (ii) holds.
\end{proof}

We omit a simple proof of the following proposition.

\begin{proposition}
	\label{s6:p9}
	$[\mathbf{ZF}]$
	Let $\mathbf{S}$ be a $c_{\delta}$-embedded (respectively, $\mathbf{2}$-embedded) subspace of a topological space $\mathbf{X}$. Let $\mathcal{F}$ be a filter in $\mathcal{CO}_{\delta}(\mathbf{S})$ (respectively, in $\mathcal{CO}(\mathbf{S})$) and let $\mathcal{U}$ be an ultrafilter in $\mathcal{CO}_{\delta}(\mathbf{X})$ (respectively, in $\mathcal{CO}(\mathbf{X})$). Then:
	\begin{enumerate} 
		\item[(i)] the family 
		\[\mathcal{F}_X=\{V\in\mathcal{CO}_{\delta}(\mathbf{X})\ (\text{respectively,}\ V\in\mathcal{CO}(\mathbf{X})): V\cap S\in\mathcal{F}\}\]
		\noindent is a filter in $\mathcal{CO}_{\delta}(\mathbf{X})$ (respectively, in $\mathcal{CO}(\mathbf{X})$);
		\item[(ii)] if, for every $U\in\mathcal{U}$, $U\cap S\neq\emptyset$, then the family 
		$$\mathcal{U}_S=\{U\cap S: U\in\mathcal{U}\}$$
		is an ultrafilter in $\mathcal{CO}_{\delta}(\mathbf{S})$ (respectively, in $\mathcal{CO}(\mathbf{S})$).  
	\end{enumerate} 
\end{proposition}

\begin{theorem}
	\label{s6:t10}
	$[\mathbf{ZF}]$ Let $\mathbf{X}$ be a zero-dimensional  $T_1$-space. Let $\{\mathbf{S}_n: n\in\omega\}$ be a family of $c_{\delta}$-embedded $\mathbb{N}$-compact subspaces of $\mathbf{X}$ such that $X=\bigcup_{n\in\omega}S_n$. Then $\mathbf{CMC}$ implies that $\mathbf{X}$ is $\mathbb{N}$-compact. 
\end{theorem}
\begin{proof}
	Assuming $\mathbf{CMC}$, we consider any $c_{\delta}$-ultrafilter in $\mathbf{X}$ such that $\mathcal{U}$ has the countable intersection property. It follows from Proposition \ref{s6:p8}(i) that $\mathcal{U}$ is closed under countable intersections. In much the same way, as in the proof of Theorem \ref{s6:t6}, we can show that there exists $n_0\in\omega$ such that, for every $U\in\mathcal{U}$, $U\cap S_{n_0}\neq\emptyset$. By Proposition \ref{s6:p9}(ii), the family $\mathcal{U}_{n_0}=\{U\cap S_{n_0}: U\in\mathcal{U}\}$ is a $c_{\delta}$-ultrafilter in $\mathbf{S}_{n_0}$. Since $\mathcal{U}$ is closed under countable intersections, $\mathbf{CMC}$ implies that $\mathcal{U}_{n_0}$ is also closed under countable intersections. Since $\mathbf{S}_{n_0}$ is $\mathbb{N}$-compact, it follows from Theorem \ref{s4:t2} that $\mathcal{U}_{n_0}$ is fixed. This implies that $\mathcal{U}$ is fixed, so $\mathbf{X}$ is $\mathbb{N}$-compact by Theorem \ref{s4:t2}.
\end{proof}

\begin{corollary}
	\label{s6:c11}
	$[\mathbf{ZF}]$ Let $\mathbf{X}$ be a zero-dimensional $T_1$-space. Let $\{\mathbf{S}_n: n\in\omega\}$ be a family of $\mathbf{2}$-embedded $\mathbb{N}$-compact subspaces of $\mathbf{X}$ such that $X=\bigcup_{n\in\omega}S_n$. Then $\mathbf{CMC}$ implies that $\mathbf{X}$ is $\mathbb{N}$-compact. 
\end{corollary}
\begin{proof}
	It suffices to apply Theorem \ref{s6:t10} and Proposition \ref{s6:p8}(ii). 
\end{proof}

Theorem 8.16 of \cite{gj} and Fact 21 of \cite{nik} assert that every Tychonoff space expressible as the union of a realcompact subspace and a compact subspace is realcompact. The proof of it in \cite{gj} is a proof in $\mathbf{ZF}$ and can be easily modified to a $\mathbf{ZF}$-proof of Theorem \ref{s6:t12} given below; however, the axiom of choice is involved in its alternative proof in \cite{nik}. This is why, by applying our Theorems \ref{s3:t2} and \ref{s5:t3}, we are going to show how to modify the $\mathbf{ZFC}$-proof of Fact 21 in \cite{nik} to an alternative $\mathbf{ZF}$-proof of the following theorem.

\begin{theorem}
	\label{s6:t12}
	$[\mathbf{ZF}]$ Let $\mathbf{X}$ be a $T_1$-space. Suppose that $X=S\cup K$ where the set $K$ is compact in $\mathbf{X}$. Then the following conditions are satisfied:
	\begin{enumerate}
		\item[(i)](\cite[Theorem 8.16]{gj}) if $\mathbf{X}$ is completely regular and the subspace $\mathbf{S}$ of $\mathbf{X}$ is realcompact, then $\mathbf{X}$ is realcompact;
		\item[(ii)] if $\mathbf{X}$ is zero-dimensional and the subspace $\mathbf{S}$ of $\mathbf{X}$ is $\mathbb{N}$-compact, then $\mathbf{X}$ is $\mathbb{N}$-compact. 
	\end{enumerate}
\end{theorem}

\begin{proof} Let $\mathbf{X}$ be a Tychonoff space (respectively, a zero-dimensional space). Suppose that $\mathcal{F}$ is a free ultrafilter in $\mathcal{Z}(\mathbf{X})$ (respectively, in $\mathcal{CO}(\mathbf{X})$). Put
	$$\mathcal{F}_S= \{S\cap F: F\in\mathcal{F}\}.$$
	\noindent Since $\bigcap\mathcal{F}=\emptyset$ and $K$ is compact in $\mathbf{X}$, we can fix $Z_0\in\mathcal{F}$ such that $K\cap Z_0=\emptyset$. Now, one can easily check that $\mathcal{F}_S$ is a filter in $\mathcal{Z}(\mathbf{S})$ (respectively, in $\mathcal{CO}(\mathbf{S})$). Let us assume that $W\in\mathcal{Z}(\mathbf{S})$.
	
	(i) Assuming that $\mathbf{X}$ is a Tychonoff space, by the compactness of $K$ in $\mathbf{X}$, we can choose $Z_1\in\mathcal{Z}(\mathbf{X})$ such that $K\subseteq \inter_{\mathbf{X}}(Z_1)$ and $Z_0\cap Z_1=\emptyset$. In much the same way, as in \cite[proof of Fact 21]{nik}, one can show that $W\cap Z_0\in\mathcal{Z}(\mathbf{X})$. Suppose that, for every $T\in\mathcal{F}_S$, $W\cap T\neq\emptyset$. Then, for every $F\in\mathcal{F}$, $W\cap Z_0\cap F\neq\emptyset$, so $W\cap Z_0\in\mathcal{F}$. Hence $W\in\mathcal{F}_S$. This proves that $\mathcal{F}_S$ is an ultrafilter in $\mathcal{Z}(\mathbf{S})$. 
	
	Now, assume that $\mathcal{F}$ has the weak countable intersection property. Let us show that $\mathcal{F}_S$ has the weak countable intersection property. To this aim, we assume that  $\{f_n: n\in\omega\}$ is a family of functions from $C(\mathbf{S}, [0, 1])$ such that, for every $n\in\omega$, $Z(f_n)\in\mathcal{F}_S$. We can fix $g\in C(\mathbf{S}, [0, 1])$ such that $S\cap Z_1=Z(g)$. Using the idea from \cite[proof of Fact 21]{nik}, for every $n\in\omega$, we define a function $h_n\in C(\mathbf{X}, [0, 1])$ as follows:
	\[
	h_n(x)=\begin{cases} \frac{f_n(x)}{f_n(x)+g(x)} &\quad\text{if}\quad x\in X\setminus \inter_{\mathbf{X}}Z_1;\\
		1 &\quad\text{if}\quad x\in Z_1.\end{cases}
	\]
	We notice that, for every $n\in\omega$, $Z(f_n)\cap Z_0\in\mathcal{F}$ and $Z(f_n)\cap Z_0=Z(h_n)$. Since $\mathcal{F}$ has the weak countable intersection property, $\bigcap_{n\in\omega}Z(h_n)\neq\emptyset$. In consequence, $\mathcal{F}_S$ has the weak countable intersection property. Since $\mathcal{F}_S$ is free, $\mathbf{S}$ is not realcompact by Theorem \ref{s5:t3}. This proves (i).
	
	(ii) Now, we assume that $\mathbf{X}$ is zero-dimensional, $\mathcal{F}$ is a free ultrafilter in $\mathcal{CO}(\mathbf{X})$, and $Z_0\in\mathcal{F}$ is such that $Z_0\cap K=\emptyset$.  We notice that if $W\in\mathcal{CO}(\mathbf{S})$, then,  $W\cap Z_0\in \mathcal{CO}(\mathbf{X})$, so, arguing in much the same way, as in the proof of (i), we can show that $\mathcal{F}_S$ is an ultrafilter in $\mathcal{CO}(\mathbf{S})$. Assuming that $\mathcal{F}$ has the countable intersection property, to check that $\mathcal{F}_S$ has the countable intersection property, it suffices to notice that if $\{E_n: n\in\omega\}$ is a collection of members of $\mathcal{F}_S$, then $\{E_n\cap Z_0: n\in\omega\}$ is a collection of members of $\mathcal{F}$. Hence, by Theorem \ref{s3:t2}, if $\mathbf{S}$ is $\mathbb{N}$-compact, so is $\mathbf{X}$.  
\end{proof}

\begin{definition}
	\label{s6:d13}
	A topological space $\mathbf{X}$ is called \emph{hereditarily realcompact} (respectively, \emph{hereditarily} $\mathbb{N}$-\emph{compact}) if every subspace of $\mathbf{X}$ is realcompact (respectively, $\mathbb{N}$-compact).
\end{definition}

The first characterizations of hereditary realcompactness in $\mathbf{ZFC}$ were given by Shirota in \cite[Theorem 6]{sh}, and slightly developed in \cite[Theorem 8.17]{gj} (see also  \cite[Exercise 3.11 B]{en}). Let us strengthen and extend Theorem 8.17 of \cite{gj} (cf. \cite[Exercise 3.11 B]{en}) as follows. 

\begin{theorem}
\label{s6:t14}
$[\mathbf{ZF}]$
For every completely regular (respectively, zero-dimen\-sional) $T_1$-space $\mathbf{X}$, the following conditions are equivalent:
\begin{enumerate}
\item[(i)] $\mathbf{X}$ is hereditarily realcompact (respectively, hereditarily $\mathbb{N}$-compact);
\item[(ii)] for every $x\in X$, the subspace $X\setminus\{x\}$ of $\mathbf{X}$ is realcompact (respectively, $\mathbb{N}$-compact);
\item[(iii)] every completely regular (respectively, zero-dimensional) $T_1$-space that can be mapped onto $\mathbf{X}$ by a continuous mapping with compact fibres is realcompact (respectively, $\mathbb{N}$-compact);
\item[(iv)] every completely regular (respectively, zero-dimensional) $T_1$-space that can be mapped onto $\mathbf{X}$ by a continuous injective mapping is realcompact (respectively, $\mathbb{N}$-compact).
\end{enumerate}
\end{theorem}
\begin{proof}
With Propositions \ref{s2:p14}, \ref{s2:p15} and Theorem \ref{s6:t12} in hand, it suffices to mimic the proof of Theorem 8.17 in \cite{gj} and slightly modify it for $\mathbb{N}$-compactness. Let us leave details as a simple exercise.
\end{proof}

\begin{corollary}
	\label{s6:c15}
	$[\mathbf{ZF}]$ Let $\mathbf{X}=\langle X, \tau_X\rangle$ be a $T_1$-space. Then the following hold:
	
	\begin{enumerate}
		\item[(i)] if $\mathbf{X}$ is completely regular, then $\mathbf{X}$ is hereditarily realcompact if and only if, for every completely regular topology $\tau$ on $X$ such that $\tau_X\subseteq \tau$, the space $\langle X, \tau\rangle$  is realcompact;
		\item[(ii)] if $\mathbf{X}$ is zero-dimensional, then  $\mathbf{X}$ is hereditarily $\mathbb{N}$-compact if and only if, for every zero-dimensional topology $\tau$ on $X$ such that $\tau_X\subseteq \tau$, the space $\langle X, \tau\rangle$  is $\mathbb{N}$-compact.
	\end{enumerate}
\end{corollary}

\begin{remark}
	\label{s6:r16} 
	In \cite[Remark (iii), p.  234]{lr}, there is the following statement: ``a topology stronger than a realcompact topology is realcompact''. However, it follows from our Corollary \ref{s6:c15} that,  in $\mathbf{ZF}$, a Tychonoff topology stronger than a realcompact topology need not be realcompact, and a zero-dimensional topology stronger than an $\mathbb{N}$-compact topology need not be $\mathbb{N}$-compact.
\end{remark}

To show a reasonable sufficient condition for a stronger topology than a given realcompact (respectively, $\mathbb{N}$-compact) topology to be realcompact (respectively, $\mathbb{N}$-compact), let us give the following definitions and prove two lemmas.

\begin{definition}
	\label{s6:d17}
	Let $\mathbf{X}$ and $\mathbf{Y}$ be topological spaces. A continuous surjection $f: \mathbf{X}\to\mathbf{Y}$ is called $z$-\emph{perfect} (respectively, $c_{\delta}$-\emph{perfect}) if, for every $p\in Y$, the set $f^{-1}[\{p\}]$ is compact in $\mathbf{X}$  and, for every closed set $C$ in $\mathbf{X}$ and every $y\in Y\setminus f[C]$, there exists $Z\in \mathcal{Z}(\mathbf{Y})$ (respectively, $Z\in\mathcal{C}_{\delta}(\mathbf{Y})$) such that $y\in Z\subseteq Y\setminus f[C]$.
\end{definition}

\begin{definition}
\label{s6:d18}
Let $\mathbf{X}=\langle X, \tau\rangle$ be a topological space. 
	\begin{enumerate}
		\item[(i)] (Cf., e.g., \cite{lr} and \cite{lorch}.) The \emph{Baire topology} of $\mathbf{X}$ is the topology $\tau_z$ on $X$ such that $\mathcal{Z}(\mathbf{X})$ is a base of $\langle X, \tau_z\rangle$. 
		\item[(ii)] The $c_{\delta}$-\emph{topology} of $\mathbf{X}$ is the topology $\tau_{c_\delta}$ on $X$ such  that $\mathcal{CO}_{\delta}(\mathbf{X})$ is a base of $\langle X, \tau_{c_\delta}\rangle$. 
	\end{enumerate}
\end{definition}

\begin{lemma}
	\label{s6:l19}
	$[\mathbf{ZF}]$ Let $C$ be a subset of a topological space $\mathbf{X}$. Then $C\in \mathcal{CO}_{\delta}(\mathbf{X})$ if and only if there exists $f\in C(\mathbf{X},\mathbb{N}(\infty))$ such that $C=f^{-1}[\{\infty\}]$.
\end{lemma}
\begin{proof}
	Clearly, if $f\in C(\mathbf{X}, \mathbb{N}(\infty))$ and $C=f^{-1}[\{\infty\}]$, then $C$ is a $c_{\delta}$-set of $\mathbf{X}$. On the other hand, if $C$ is clopen in $\mathbf{X}$, we define a function $g\in C(\mathbf{X}, \mathbb{N}(\infty))$ as follows:
	$$
	g(x)=\begin{cases} 1 &\quad\text{if}\quad x\in X\setminus C;\\
		\infty &\quad\text{if}\quad x\in C.\end{cases}
	$$
	\noindent Then $C=g^{-1}[\{\infty\}]$. 
	
	Let us assume that $C$ is not clopen in $\mathbf{X}$ but $C\in\mathcal{CO}_{\delta}(\mathbf{X})$. Then we can fix a family $\{C_n: n\in\omega\}$ of clopen sets of $\mathbf{X}$ such that  $C=\bigcap_{n\in\omega}C_n$ where $C_0=X$ and, for every $n\in\omega$, $C_{n+1}\subsetneq C_n$. Let us define a function $h: X\to\mathbb{N}(\infty)$ as follows:
	\[
	h(x)=\begin{cases} n+1 &\quad\text{if}\quad x\in C_n\setminus C_{n+1};\\
		\infty &\quad\text{if}\quad x\in C.\end{cases}
	\]
	\noindent Then $h\in C(\mathbf{X}, \mathbb{N}(\infty))$ and $C=h^{-1}[\{\infty\}]$. 
\end{proof}

\begin{lemma}
	\label{s6:l20}
	$[\mathbf{ZF}]$ Suppose that $\mathbf{X}$ and $\mathbf{Y}$ are completely regular (respectively, zero-dimension\-al) $T_1$-spaces, and $f:\mathbf{X}\to\mathbf{Y}$ is a $z$-perfect (respectively $c_{\delta}$-perfect) mapping. Let $\tilde{f}: v\mathbf{X}\to v\mathbf{Y}$ (respectively, $\tilde{f}:v_0\mathbf{X}\to v_{0}\mathbf{Y}$) be the continuous extension of $f$. Then $\tilde{f}[vX\setminus X]\subseteq vY\setminus Y$ (respectively, $\tilde{f}[v_{0}X\setminus X]\subseteq v_{0}Y\setminus Y$).
\end{lemma}
\begin{proof}
	
	Suppose that $p\in vX\setminus X$ (respectively, $p\in v_{0}X\setminus X$) is such that $\tilde{f}(p)\in Y$. There exists $x_p\in X$ such that $\tilde{f}(p)=f(x_p)$. The set $K=f^{-1}[f(x_p)]$ is compact in $\mathbf{X}$, so there exist sets $Z_1, Z_2\in\mathcal{Z}(v\mathbf{X})$ (respectively, $Z_1, Z_2\in \mathcal{CO}(v_{0}\mathbf{X})$) such that $p\in Z_1$, $K\subseteq Z_2$ and $Z_1\cap Z_2=\emptyset$. Since $f(x_p)\notin f[X\cap Z_1]$, there exists $Z_3\in\mathcal{Z}(\mathbf{Y})$ (respectively, $Z_3\in\mathcal{CO}_{\delta}(\mathbf{Y})$) such that $f(x_p)\in Z_3\subseteq Y\setminus f[X\cap Z_1]$. There exists $g\in C(v\mathbf{Y}, [0, 1])$ (respectively, $G\in \mathcal{CO}_{\delta}(v_{0}\mathbf{Y})$) such that $Z_3=Z(g)\cap Y$ (respectively, $Z_3=G\cap Y$). Let $Z_4= Z(g\circ \tilde{f})$ (respectively, $Z_4=\tilde{f}^{-1}[G]$).  Then $p\in Z_4\in\mathcal{Z}(v\mathbf{X})$ (respectively, $p\in Z_4\in\mathcal{CO}_{\delta}(v_{0}\mathbf{X})$) and $Z_4\cap Z_1\cap X=\emptyset$. Hence, for $Z_5=Z_4\cap Z_1$, we have $Z_5\in\mathcal{Z}(v\mathbf{X})$ (respectively, $Z_5\in\mathcal{CO}_{\delta}(v_{0}\mathbf{X})$) and $\emptyset\neq Z_5\subseteq vX\setminus X$ (respectively, $\emptyset\neq Z_5\subseteq v_{0}X\setminus X$). Let $S=vX\setminus Z_5$ (respectively, $S=v_{0}X\setminus Z_5$). It follows from Proposition \ref{s2:p15} that if $Z_5\in\mathcal{Z}(v\mathbf{X})$, then the subspace $\mathbf{S}$ of $v\mathbf{X}$ is realcompact. By Proposition \ref{s2:p15} and Lemma \ref{s6:l19}, if $Z_5\in\mathcal{CO}_{\delta}(v_0\mathbf{X})$, the subspace $\mathbf{S}$ of $v_{0}\mathbf{X}$ is $\mathbb{N}$-compact. Therefore, since $X\subseteq S\subseteq vX$ (respectively, $X\subseteq S\subseteq v_{0}X$), the equality $S=vX$ (respectively, $S=v_{0}X$) holds. But this is impossible because $Z_5\neq\emptyset$. The contradiction obtained completes the proof.
\end{proof}

\begin{theorem}
	\label{s6:t21}
	$[\mathbf{ZF}]$ Let $\mathbf{X}=\langle X, \tau\rangle$ be a realcompact $($respectively, $\mathbb{N}$-compact$)$ space. Let $\tau^{\ast}$ be a topology on $X$ such that $\langle X, \tau^{\ast}\rangle$ is completely regular $($respectively, zero-dimensional$)$ and $\tau\subseteq\tau^{\ast}\subseteq \tau_{z}$ $($respectively,  $\tau\subseteq\tau^{\ast}\subseteq \tau_{c_{\delta}}$$)$. Then $\langle X, \tau^{\ast}\rangle$ is realcompact $($respectively, $\mathbb{N}$-compact$)$.
\end{theorem}
\begin{proof}
	It suffices to observe that the identity map $\id_{X}:\langle X, \tau^{\ast}\rangle\to\langle X, \tau\rangle$ is $z$-perfect (respectively, $c_{\delta}$-perfect), so $\langle X, \tau^{\ast}\rangle$ is realcompact (respectively, $\mathbb{N}$-compact) by Lemma \ref{s6:l20}.
\end{proof}

Let make comments on Corollary 8.15 in \cite{gj} asserting that if $\mathbf{X}$ is a realcompact space whose every singleton is of type $G_{\delta}$, then $\mathbf{X}$ is hereditarily realcompact. We do not know if this corollary is provable in $\mathbf{ZF}$; however, we can obtain the following modifications of it in $\mathbf{ZF}$.

\begin{proposition}
\label{s6:p22}
$[\mathbf{ZF}]$
\begin{enumerate}
\item[(i)] If $\mathbf{X}$ is a realcompact space such that, for every $x\in X$, $\{x\}\in\mathcal{Z}(\mathbf{X})$, then $\mathbf{X}$ is hereditarily realcompact.
\item[(ii)] If $\mathbf{X}$ is an $\mathbb{N}$-compact space such that, for every $x\in X$, $\{x\}\in \mathcal{CO}_{\delta}(\mathbf{X})$, then $\mathbf{X}$ is hereditarily $\mathbb{N}$-compact.
\end{enumerate}
\end{proposition}
\begin{proof}
We omit the proof of (i) because it is the same, as the proof of Corollary 8.15 in \cite{gj}. 

(ii) Let us assume that $\mathbf{X}$ is an $\mathbb{N}$-compact space such that, for every $x\in X$, $\{x\}\in\mathcal{CO}_{\delta}(\mathbf{X})$. Fix a point $x_0\in X$. It follows from Lemma \ref{s6:l19} that there exists $f\in C(\mathbf{X}, \mathbb{N}(\infty))$ such that $\{x_0\}=f^{-1}[\{\infty\}]$. Then $X\setminus\{x_0\}=f^{-1}[\mathbb{N}]$, so the subspace $X\setminus\{x_0\}$ of $\mathbf{X}$ is $\mathbb{N}$-compact by Proposition \ref{s2:p15}. Theorem \ref{s6:t14} completes the proof.
\end{proof}

\begin{proposition}
\label{s6:p23}
$[\mathbf{ZF+CMC}]$
\begin{enumerate}
\item[(i)] If $\mathbf{X}$ is a realcompact space whose every singleton is of type $G_{\delta}$ in $\mathbf{X}$, then $\mathbf{X}$ is hereditarily realcompact.
\item[(ii)] If $\mathbf{X}$ is an $\mathbb{N}$-compact space whose every singleton is of type $G_{\delta}$ in $\mathbf{X}$, then $\mathbf{X}$ is hereditarily $\mathbb{N}$-compact.
\end{enumerate}
\end{proposition}

\begin{proof}
Let $\mathbf{X}$ be a completely regular (respectively, zero-dimensional) $T_1$-space and let $x_0\in X$ be such that the singleton $\{x_0\}$ of type $G_{\delta}$ in $\mathbf{X}$. Let $\{G_n: n\in\omega\}$ be a family of open sets in $\mathbf{X}$ such that $\{x_0\}=\bigcap_{n\in\omega}G_n$. For every $n\in \omega$, let $\mathcal{A}_n=\{A\in\mathcal{Z}(\mathbf{X}): x_0\in A\subseteq G_n\}$ (respectively, $\mathcal{A}_n=\{A\in\mathcal{CO}(\mathbf{X}): x_0\in A\subseteq G_n\}$). By $\mathbf{CMC}$, there exists a family $\{\mathcal{C}_n: n\in\omega\}$ of non-empty finite sets such that, for every $n\in\omega$, $\mathcal{C}_n\subseteq \mathcal{A}_n$. For every $n\in\omega$, let $C_n=\bigcap\mathcal{C}_n$. Then, for every $n\in\omega$, $C_n\in\mathcal{Z}(\mathbf{X})$ (respectively, $C_n\in\mathcal{CO}(\mathbf{X})$). Clearly $\{x_0\}=\bigcap_{n\in\omega}C_n$. To complete the proof, it suffices to apply Proposition \ref{s6:p22} and Lemma \ref{s4:l1} or Lemma \ref{s5:l1}(ii).
\end{proof}

 We recall that it was proved in \cite[Theorem 3.6]{kow} that a zero-dimensional subspace of $\mathbb{R}$ may fail to be strongly zero-dimensional in $\mathbf{ZF}$. Therefore, the following theorem is non-trivial.

\begin{theorem}
\label{s6:t24}
$[\mathbf{ZF}]$
Every zero-dimensional $F_{\sigma}$-subspace of $\mathbb{R}$ has a countable base consisting of clopen sets, is strongly zero-dimensional and $\mathbb{N}$-compact.
\end{theorem}
\begin{proof}
Let $X$ be a non-empty $F_{\sigma}$-set in $\mathbb{R}$ such that the subspace $\mathbf{X}$ of $\mathbb{R}$ is zero-dimensional. Let $Y=\mathbb{R}\setminus Y$. It was shown in \cite[Theorem 2.1(i)]{kw2} that every $G_{\delta}$-subspace of $\mathbb{R}$ is separable in $\mathbf{ZF}$. Let $D$ be a countable dense set in $\mathbf{Y}$. Since $\inter_{\mathbb{R}}(X)=\emptyset$, it is eaasily seen that the family $\{(a, b)\cap X: a,b\in D\ \text{and }\ a<b\}$ is a countable base of clopen sets for $\mathbf{X}$. It was observed in \cite[Prroposition 3.5]{kow} that it holds in $\mathbf{ZF}$ that if a topological space has a countable base consisting of clopen sets, then it is strongly zero-dimensional. This is why $\mathbf{X}$ is strongly zero-dimensional. It remains to prove that $\mathbf{X}$ is $\mathbb{N}$-compact in $\mathbf{ZF}$.

 Let $\mathcal{U}$ be a clopen ultrafilter in $\mathbf{X}$ with the countable intersection property. Let $\{K_n: n\in\omega\}$ be a family of compact subsets of $\mathbb{R}$ such that $X=\bigcup_{n\in\omega}K_n$.  Let $\mathcal{B}=\{B_m: m\in\omega\}$ be a countable base of clopen sets for $\mathbf{X}$. We may assume that, for  every $m\in\omega$, $X\setminus B_m\in\mathcal{B}$. We may also assume that $\mathcal{B}$ is closed under finite intersections.  Suppose that, for every $n\in\omega$, the family $\mathcal{C}_n=\{C\in \mathcal{U}: C\cap K_n=\emptyset\}$ is non-empty. We fix $n\in\omega$ and $C\in\mathcal{C}_n$. Then $C$ is a countable intersection of some members of $\mathcal{B}$. Since $C\cap K_n=\emptyset$, it follows from the compactness of $K_n$ that there exists $B\in\mathcal{B}$ such that  $C\subseteq B$ and $B\cap K_n=\emptyset$. Hence, the set $A_n=\{m\in\omega: B_m\in\mathcal{U}\wedge B_m\cap K_n=\emptyset\}$ is non-empty.  We can define in $\mathbf{ZF}$ the number $m(n)=\min  A_n$ . Then we obtain the family $\{B_{m(n)}: n\in\omega\}$ of members of $\mathcal{U}$ such that $\bigcap_{n\in\omega} B_{m(n)}=\emptyset$. This contradicts the countable intersection property of $\mathcal{U}$. The contradiction obtained shows that there exists $n_0\in\omega$ such that $\mathcal{C}_{n_0}=\emptyset$. By the compactness of $K_{n_0}$, $K_{n_0}\cap\bigcap\mathcal{U}\neq\emptyset$. This, together with Theorem \ref{s3:t2}, shows that $\mathbf{X}$ is $\mathbb{N}$-compact.
\end{proof}

\section{The list of open problems}
\label{s7}

\begin{problem}
\label{s7:1}
 May a zero-dimensional subspace of $\mathbb{R}$ fail to be $\mathbb{N}$-compact in a model of $\mathbf{ZF}$? (See Theorem \ref{s6:t24}.)
\end{problem}
\begin{problem}
\label{s7:2}
Is it provable in $\mathbf{ZF}$ that if every $c_{\delta}$-ultrafilter with the countable intersection property in a zero-dimensional space $\mathbf{X}$ is fixed, then $\mathbf{X}$ is $\mathbb{N}$-compact? (See Theorem \ref{s4:t2}.)
\end{problem}
\begin{problem}
\label{s7:3}
 Is it provable in $\mathbf{ZF}$ that if every $z$-ultrafilter with the countable intersection property in a Tychonoff space $\mathbf{X}$ is fixed, then $\mathbf{X}$ is realcompact? (See Theorem \ref{s5:t3} and Corollary \ref{s5:c4}.)
\end{problem}
\begin{problem}
\label{s7:4}
Is it provable in $\mathbf{ZF}$ that every Tychonoff Lindel\"of space is realcompact? (See Theorem \ref{s6:t2}.)
\end{problem}
\begin{problem}
\label{s7:5}
 Is it provable in $\mathbf{ZF}$ that every strongly zero-dimensional realcompact space is $\mathbb{N}$-compact? (See Theorem \ref{s6:t5}.)
 \end{problem}
\begin{problem}
\label{s7:6}
Is it provable in $\mathbf{ZF}$ that if a Tychonoff space $\mathbf{X}$ is a countable union of realcompat $z$-embedded subspaces, then $\mathbf{X}$ is realcompact? (See Theorem \ref{s6:t6}.)
\end{problem}
\begin{problem}
\label{s7:7}
Is it provable in $\mathbf{ZF}$ that if a zero-dimensional $T_1$-space is a countable union of $c_{\delta}$-embedded $\mathbb{N}$-compact subspaces, then $\mathbf{X}$ is $\mathbb{N}$-compact? (See Theorem \ref{s6:t10}.)
\end{problem}
\begin{problem}
\label{s7:8}
Is it provable in $\mathbf{ZF}$ that if a zero-dimensional $T_1$-space is a countable union of $\mathbf{2}$-embedded $\mathbb{N}$-compact subspaces, then $\mathbf{X}$ is $\mathbb{N}$-compact? (See Corollary \ref{s6:c11}.)
\end{problem}
\begin{problem}
\label{s7:9}
Is Proposition \ref{s6:p23} provable in $\mathbf{ZF}$?
\end{problem}

\section*{Acknowledgements}

The first author declares that the research funding for this project was provided by the Institute for Research in Fundamental Sciences (IPM) Grants Committee (Award No. 1403030024).

\end{document}